\documentclass[11pt]{article}

\usepackage{array}
\usepackage{amsfonts}
\usepackage{amscd}
\usepackage{amssymb}
\usepackage{amsthm}
\usepackage{amsmath}
\usepackage{stmaryrd}
\usepackage{graphicx}
\usepackage{color}
\usepackage{verbatim}
\usepackage{mathrsfs}
\usepackage{tikz}
\usetikzlibrary{calc}
\usepackage{caption}
\usepackage[neveradjust]{paralist}
\usepackage{subfigure}
\usepackage[a4paper, total={15cm, 23cm},centering]{geometry}
\usepackage{float}

 \theoremstyle{plain}
 \newtheorem{thm1}{Theorem}
 \newtheorem{cor1}[thm1]{Corollary}
\newtheorem{thm}{Theorem}[section]
\newtheorem{lemma}[thm]{Lemma}
\newtheorem{prop}[thm]{Proposition}
\newtheorem{cor}[thm]{Corollary}

\theoremstyle{definition}

\newtheorem{remark}[thm]{Remark}

\newtheorem{convention}[thm]{Convention}
\numberwithin{equation}{section}

\def\sD{\mathsf{D}}
\def\sE{\mathsf{E}}
\def\sF{\mathsf{F}}
\def\sG{\mathsf{G}}

\def\cL{\mathcal{L}}

\def \cP{\mathcal{P}}

\newcommand{\pperp}{\perp\hspace{-0.15cm}\perp}

\newcommand{\ssT}{\mathsf{T}}

\newcommand{\ssN}{\mathsf{N}}

\def\EE{\mathbb{E}}
\def\FF{\mathbb{F}}

\def\JJ{\mathbb{J}}
\def\KK{\mathbb{K}}
\def\LL{\mathbb{L}}

\def\J{\mathbb{J}}

\DeclareMathOperator\kar{\mathrm{char}}

\def\<{\langle}
\def\>{\rangle}

\makeatletter
\renewcommand{\@makefnmark}{\mbox{\textsuperscript{}}}
\makeatother

\title{Automorphisms and opposition in spherical buildings of exceptional type, II: Moufang hexagons}
\author{James Parkinson\footnote{This work was partially supported by the Australian Research Council Discovery Project~DP200100712.} 
\and
Hendrik Van Maldeghem}
\date{\today}

\begin{document}

\maketitle

\begin{abstract}
We classify the automorphisms of a Moufang hexagon mapping no chamber to an opposite chamber (such automorphisms are called \textit{domestic}). This forms part of a larger program to classify domestic automorphisms of Moufang spherical buildings.
\end{abstract}

\begin{center}
\textit{Dedicated to Jacques Tits}
\end{center}

\section*{Introduction}

An automorphism of a spherical building is called \textit{domestic} if it maps no chamber to an opposite chamber. Recently a systematic investigation of domestic automorphisms has revealed a beautiful connection between domesticity and large rich fixed element structures of the automorphism, and there are now complete classifications of the domestic automorphisms for various classes of spherical buildings. For example, by \cite{HVM:12} the domestic dualities of $\sE_6(\FF)$ buildings with $|\FF|>2$ are precisely the polarities that fix a split building of type $\sF_4$, and by \cite{HVM:13} the domestic trialities of thick $\sD_4$ buildings are precisely the automorphisms fixing a split building of type $\sG_2$. Moreover, in~\cite{PVM:19b,PVM:21} we classified the domestic automorphisms of split spherical buildings of types $\sE_6$, $\sF_4$, and $\sG_2$, as well as providing partial classifications in the $\sE_7$ and $\sE_8$ cases. 

The case of rank~$2$ spherical buildings (equivalently, generalised polygons) is complicated by the lack of classification of such buildings, which makes a complete classification of domestic automorphisms of arbitrary generalised polygons impossible. However \textit{Moufang} generalised polygons have been classified by Tits and Weiss~\cite{TW:02}. In the case of Moufang hexagons the classification was announced in~\cite{Tit:76}, with the complete proof appearing in~\cite{TW:02}. The classification is given in terms of \textit{hexagonal systems} (which in turn were classified by Petersson and Racine~\cite{PR:86}). In this paper we give the complete classification of domestic automorphisms of Moufang hexagons.

It is easy to see that no duality of a generalised hexagon is domestic, and so we restrict attention to collineations (that is, type preserving automorphisms). A nontrivial collineation can be domestic for one of three reasons: either it maps no point to an opposite point (\textit{point-domestic}), or no line to an opposite line (\textit{line-domestic}), or it maps both points and lines to opposite points and lines yet maps no chamber (that is, incident point-line pair) to an opposite chamber (\textit{exceptional domestic}). As a (very) special case of a result of Abramenko and Brown~\cite{AB:09} the first two possibilities are mutually exclusive. 

\goodbreak

Our main theorem is as follows. By convention we fix the duality class of Moufang hexagons so that if $\Gamma$ is associated to the hexagonal system $(\JJ,\FF,\#)$ then the points on a line are indexed by $\{\infty\}\cup\FF$ and the lines through a point are indexed by $\{\infty\}\cup\JJ$. For example, if $\Gamma$ is finite then it has parameters $(s,t)=(|\FF|,|\JJ|)$.

\begin{thm1}\label{thm:main}
Let $\Gamma$ be a Moufang hexagon, with the above convention on the duality class. 
\begin{compactenum}[$(1)$]
\item $\Gamma$ admits a unique class of nontrivial line-domestic collineations (the long root elations).
\item $\Gamma$ admits a nontrivial point-domestic collineation if and only if $\Gamma$ is a dual split Cayley hexagon, a mixed hexagon, or a triality hexagon of type ${^3}\sD_4$. Moreover, if $\theta$ is a nontrivial point-domestic collineation then $\theta$ has order $3$ and:
\begin{compactenum}[$(a)$]
\item if $\Gamma$ is either mixed, or is a dual split Cayley hexagon over a field of characteristic~$3$, then $\theta$ is a short root elation and there is a unique class of such collineations. 
\item if $\Gamma$ is a dual split Cayley hexagon over a field $\FF$ with $\kar\FF\neq 3$ then $\theta$ fixes an ovoid (respectively, a large full subhexagon) if $X^2+X+1$ is irreducible (respectively, reducible) over $\FF$, and in each case there is a unique class of such collineations. 
\item if $\Gamma$ is a triality hexagon of type ${^3}\sD_4$ then $\theta$ fixes a large full subhexagon, and $\theta$ is conjugate to a nontrivial element of the Galois group of the associated cubic Galois extension~$\EE/\FF$.
\end{compactenum}
\item $\Gamma$ admits an exceptional domestic collineation if and only if $\Gamma$ is a dual split Cayley hexagon over $\FF=\FF_2$ or the triality hexagon of type ${^3}\sD_4$ associated to a cubic extension of $\FF=\FF_2$. Moreover, for each of these hexagons there exists a unique class of exceptional domestic collineations, and these collineations have order~$4$.
\end{compactenum}
\end{thm1}

In \cite{PVM:19a,PVM:19b} we developed the language of opposition diagrams for automorphisms of spherical buildings. With the above convention on duality classes, the line-domestic collineations are those with opposition diagram $\sG_{2;1}^2=\begin{tikzpicture}[scale=0.5,baseline=-0.5ex]
\node [inner sep=0.8pt,outer sep=0.8pt] at (-0.5,0) (2) {$\bullet$};
\node [inner sep=0.8pt,outer sep=0.8pt] at (0.5,0) (3) {$\bullet$};
\phantom{\draw [line width=0.5pt,line cap=round,rounded corners] (2.north west)  rectangle (2.south east);}
\phantom{\draw [line width=0.5pt,line cap=round,rounded corners] (3.north west)  rectangle (3.south east);}
\draw (-0.5,0)--(0.5,0);
\draw (-0.5,0.11)--(0.5,0.11);
\draw (-0.5,-0.11)--(0.5,-0.11);
\draw (0+0.15,0.3) -- (0-0.08,0) -- (0+0.15,-0.3);
\draw [line width=0.5pt,line cap=round,rounded corners] (3.north west)  rectangle (3.south east);
\end{tikzpicture}$, and the point-domestic collineations are those with opposition diagram~$\sG_{2;1}^1=\begin{tikzpicture}[scale=0.5,baseline=-0.5ex]
\node [inner sep=0.8pt,outer sep=0.8pt] at (-0.5,0) (2) {$\bullet$};
\node [inner sep=0.8pt,outer sep=0.8pt] at (0.5,0) (3) {$\bullet$};
\phantom{\draw [line width=0.5pt,line cap=round,rounded corners] (2.north west)  rectangle (2.south east);}
\phantom{\draw [line width=0.5pt,line cap=round,rounded corners] (3.north west)  rectangle (3.south east);}
\draw (-0.5,0)--(0.5,0);
\draw (-0.5,0.11)--(0.5,0.11);
\draw (-0.5,-0.11)--(0.5,-0.11);
\draw (0+0.15,0.3) -- (0-0.08,0) -- (0+0.15,-0.3);
\draw [line width=0.5pt,line cap=round,rounded corners] (2.north west)  rectangle (2.south east);
\end{tikzpicture}$ (see \cite{PVM:21} for the notation). The above theorem immediately gives:

\begin{cor1}\label{cor:main}
Let $\Gamma$ be a Moufang hexagon, with the above convention on duality classes.
\begin{compactenum}[$(1)$]
\item There exists a collineation with opposition diagram $\sG_{2;1}^2$.
\item There exists a collineation with opposition diagram $\sG_{2;1}^1$ if and only if $\Gamma$ is a dual split Cayley hexagon, a mixed hexagon, or a triality hexagon of type ${^3}\sD_4$.
\item The hexagon $\Gamma$ admits a domestic collineation not fixing a chamber if and only if $\Gamma$  is a dual split Cayley hexagon over a field $\FF$ with $X^2+X+1$ irreducible over~$\FF$.
\end{compactenum}
\end{cor1}

The structure of this paper is as follows. In Section~\ref{sec:background} we outline background material and definitions on generalised hexagons, domesticity, and hexagonal systems. We also extend the coordinatisation of dual split Cayley hexagons and triality hexagons from \cite[Chapter~3]{HVM:98} to general Moufang hexagons (see Theorem~\ref{thm:coordinates}). By Theorem~\ref{thm:PTM3}, a collineation of a generalised hexagon is point-domestic if and only if its fixed element structure is either a ball of radius $3$ in the incidence graph centred at a line, a large full subhexagon, or an ovoid (and dually for line-domestic collineations). These three possibilities, and their duals, are each analysed in Sections~\ref{sec:elations}, \ref{sec:subhex}, and~\ref{sec:ovoids}, culminating in the proof of Theorem~\ref{thm:main} in Section~\ref{sec:main}.

\newpage

\section{Background and definitions}\label{sec:background}

In this section we begin by giving some basic definitions concerning generalised hexagons, and recalling results from the literature on domestic automorphisms of generalised hexagons. In Section~\ref{sec:Moufang} we recall the classification of Moufang hexagons in terms of hexagonal systems, following \cite{TW:02}, and record commutation relations and related formulae that will be used repeatedly throughout the paper. In Section~\ref{sec:coordinates} we extend the coordinatisation of dual split Cayley hexagons and triality hexagons from \cite[Chapter~3]{HVM:98} to general Moufang hexagons (this coordinatisation will be used in Section~\ref{sec:subhex}).

\subsection{Generalised hexagons}\label{sec:GH}

A \textit{generalised hexagon} is a nonempty point line geometry $\Gamma=(\mathcal{P},\mathcal{L})$ containing no ordinary $2$, $3$, $4$, or $5$-gon as a subgeometry such that any two elements $x,y\in\mathcal{P}\cup\mathcal{L}$ are contained in an ordinary hexagon. We will typically drop the adjective ``generalised'', and simply refer to generalised hexagons as \textit{hexagons}. 

The hexagon~$\Gamma$ is \textit{thick} if each line contains at least $3$ points, and each point is on at least $3$ lines. The \textit{distance} $d(x,y)$ between elements $x,y\in\mathcal{P}\cup\mathcal{L}$ is the distance in the incidence graph. Thus $d(x,y)\leq 6$, and elements $x,y$ are \textit{opposite} one another if and only if $d(x,y)=6$. If $x$ and $y$ are opposite then necessarily $x,y$ are either both points, or are both lines. For $p\in\cP$ let $p^{\perp}$ denote the set of all points collinear with~$p$, and write $p^{\pperp}$ for the set of all points that are not opposite~$p$.

An \textit{ovoid} of $\Gamma$ is a set $O$ of mutually opposite points such that every element $x\in\mathcal{P}\cup\mathcal{L}$ is at distance at most $3$ from some element of $O$. The dual notion of an ovoid is a \textit{spread}. 

A \textit{subhexagon} of $\Gamma$ is a subgeometry $\Gamma'$ that is itself a generalised hexagon. A subhexagon $\Gamma'$ is \textit{full} if every point of $\Gamma$ incident with a line of $\Gamma'$ belongs to $\Gamma'$, and \textit{large} if every element of $\Gamma$ is at distance at most $3$ from some element of $\Gamma'$. The dual notion to a full subhexagon is an \textit{ideal} subhexagon.

\subsection{Domestic automorphisms of generalised hexagons}\label{sec:Dom}

We now recall the known results concerning domesticity in hexagons. Firstly, it is easy to see that no duality of a thick hexagon is domestic (see \cite[Theorem~2.7]{PTM:15}). If $\theta$ is a domestic collineation of a thick hexagon~$\Gamma$ then there are three possibilities. If $\theta$ maps no point (respectively no line) to an opposite point (respectively line) then $\theta$ is called \textit{point-domestic} (respectively \textit{line-domestic}). The third possibility is that $\theta$ maps both points and lines to opposite points and lines, yet maps no chamber (that is, incident point-line pair) to an opposite. Such a collineation is called \textit{exceptional domestic}.

Exceptional domestic collineations are extremely rare, and have been completely classified for finite (that is $|\cP|,|\cL|<\infty$) thick hexagons.

\begin{thm}[{\cite[Corollary~2.11]{PTM:15}}]\label{thm:TMPex1}
If a finite thick (not necessarily Moufang) hexagon with parameters $(s,t)$ admits an exceptional domestic collineation then $(s,t)\in\{(2,2),(2,8),(8,2)\}$. Moreover, for each hexagon with these parameters there exists a unique exceptional domestic collineation up to conjugation, and these collineations have order~$4$.
\end{thm}

Moreover, the possibility of exceptional domestic collineations of infinite Moufang hexagons was also eliminated in~\cite{PTM:15} (see \cite{TW:02} for the definition of the Moufang condition).

\begin{thm}[{\cite[Theorem~2.14]{PTM:15}}]\label{thm:TMPex2} No infinite Moufang hexagon admits an exceptional domestic collineation. 
\end{thm}

Theorems~\ref{thm:TMPex1} and~\ref{thm:TMPex2} prove Theorem~\ref{thm:main}(3), and so the classification of domestic collineations of Moufang hexagons is reduced to the classification of collineations that are either point-domestic or line-domestic. The following theorem shows that such collineations are characterised by their fixed element structures. This gives an important guiding framework for our classification of domestic collineations of Moufang hexagons.

\begin{thm}[{\cite[Theorems~2.7 and 2.8]{PTM:15}}] \label{thm:PTM3} A nontrivial collineation $\theta$ of a thick generalised hexagon is point-domestic if and only if its fixed element structure is either a ball of radius $3$ in the incidence graph centred at a line, a large full subhexagon, or an ovoid. Dually, a nontrivial collineation $\theta$ of a thick generalised hexagon is line-domestic if and only if its fixed element structure is either a ball of radius $3$ in the incidence graph centred at a point, a large ideal subhexagon, or a spread.
\end{thm}

\subsection{Moufang hexagons and hexagonal systems}\label{sec:Moufang}

Let $\Gamma=(\mathcal{P},\mathcal{L})$ be a Moufang hexagon. Let $A_0$ be a fixed choice of apartment (an ordinary hexagon), and let $C_0=\{p_0,L_0\}$ be a fixed choice chamber of $A_0$. Let $G=G(\Gamma)$ denote the full collineation group of $\Gamma$. Let $B$ denote the stabiliser of $C_0$, and let $N$ denote the (set-wise) stabiliser of $A_0$. Then $(B,N)$ is a split $BN$-pair in $G$ with Weyl group $W=\langle s_1,s_6\mid s_1^2=s_6^2=(s_1s_6)^6=e\rangle$ the dihedral group of order~$12$ (see \cite[(33.4)]{TW:02}). Let $U_1,\ldots,U_{12}$ denote the root subgroups associated to $A_0$ (as in \cite{TW:02}), and let $U=\langle U_1,\ldots,U_6\rangle$ denote the subgroup of $G$ generated by the positive root subgroups. Then $B=H\ltimes U$, where $H=B\cap N$ is the element-wise stabiliser of $A_0$. Let $w_0=s_1s_6s_1s_6s_1s_6=s_6s_1s_6s_1s_6s_1$ denote the longest element of~$W$.

The Moufang hexagons are determined (up to isomorphism) by the commutator relations that hold amongst the root subgroups in $U$. In \cite{TW:02} this classiciation is given in terms of the following algebraic structures. An \textit{hexagonal system} is a triple $(\mathbb{J},\mathbb{F},\#)$, where $\mathbb{F}$ is a commutative field, $\mathbb{J}$ is a vector space over $\mathbb{F}$, and $\#:\mathbb{J}\to\mathbb{J}$ is a function called the \textit{adjoint} satisfying various axioms (see \cite[(15.15)]{TW:02}). There is a unique element $1\in \JJ\backslash\{0\}$ with $1^{\#}=1$ and we identify $\mathbb{F}$ with the subset $\{t1\mid t\in\mathbb{F}\}$ of $\mathbb{J}$. The map $\#$ determines a function $\ssN:\mathbb{J}\to\mathbb{F}$ (called the \textit{norm}), a symmetric bilinear form $\ssT:\mathbb{J}\times\mathbb{J}\to\mathbb{F}$ (called the \textit{trace}), and a symmetric bilinear map $\times:\mathbb{J}\times\mathbb{J}\to\mathbb{J}$. These maps satisfy various properties, including the following (for $a,b,c\in\JJ$ and $t\in\FF$; see \cite[(15.15), (30.4)]{TW:02}): $\ssN(1)=1$, $\ssT(1)=3$, $(ta)^{\#}=t^2a^{\#}$, $\ssN(ta)=t^3\ssN(a)$, $a\times a=2a^{\#}$,
\begin{align*}
(a+b)^{\#}&=a^{\#}+(a\times b)+b^{\#},&\ssT(a\times b,c)&=\ssT(a,b\times c),&\ssT(a,a^{\#})&=3\ssN(a)\\
a&=\ssT(a)-1\times a&\ssN(a^{\#})&=\ssN(a)^2&a^{\#\#}&=\ssN(a)a.
\end{align*}

The simplest examples are the hexagonal systems $(\FF,\FF,\#)$ with $\FF$ any field and $a^{\#}=a^2$ (and then $\ssN(a)=a^3$, $\ssT(a,b)=3ab$, and $a\times b=2ab$), and the hexagonal systems $(\EE,\FF,\#)$ where $\EE/\FF$ is a separable cubic extension and $a^{\#}=a^{\sigma}a^{\sigma^2}$ with $\sigma$ a nontrivial element of $\mathrm{Gal}(\LL/\FF)$ where $\LL/\FF$ is the normal closure of $\EE/\FF$ (and then $\ssN(a)=aa^{\sigma}a^{\sigma^2}$, $\ssT(a,b)=ab+a^{\sigma}b^{\sigma}+a^{\sigma^2}b^{\sigma^2}$, and $a\times b=a^{\sigma}b^{\sigma^2}+a^{\sigma^2}b^{\sigma}$). The complete list of hexagonal systems is given in~\cite[(15.14)]{TW:02}.

By a \textit{subhexagonal system} of $(\JJ,\FF,\#)$ we shall mean a triple $(\JJ',\FF,\#)$ with $\JJ'$ a subspace of $\JJ$ closed under $\#$ (such systems are called \textit{substructures} in \cite{TW:02}). 

By the classification of Tits and Weiss~\cite[Theorem 17.5]{TW:02} every Moufang hexagon arises from an hexagonal system $(\mathbb{J},\mathbb{F},\#)$ via the construction in~\cite[(16.8)]{TW:02}. The nontrivial commutator relations amongst the groups $U_1,\ldots,U_6$ are as follows (where $[g,h]=g^{-1}h^{-1}gh$ and $a,b\in \mathbb{J}$ and $t,u\in\mathbb{F}$; see \cite[(16.8)]{TW:02})
\begin{align}
\label{eq:commutator1}[x_1(a),x_3(b)]&=x_2(\ssT(a,b))\\
\label{eq:commutator2}[x_3(a),x_5(b)]&=x_4(\ssT(a,b))\\
\label{eq:commutator3}[x_1(a),x_5(b)]&=x_2(-\ssT(a^{\#},b))x_3(a\times b)x_4(\ssT(a,b^{\#}))\\
\label{eq:commutator4}[x_2(t),x_6(u)]&=x_4(tu)\\
\label{eq:commutator5}[x_1(a),x_6(t)]&=x_2(-t\ssN(a))x_3(ta^{\#})x_4(t^2\ssN(a))x_5(-ta).
\end{align}
We have $U_i=\langle x_i(a)\mid a\in \JJ\rangle$ if $i$ is odd, and $U_i=\langle x_i(t)\mid t\in \FF\rangle$ if $i$ is even. Each $U_i$ is abelian and $x_i(a)x_i(b)=x_i(a+b)$ (for $a,b\in\JJ$ if $i$ is odd, and $a,b\in\FF$ if $i$ is even). Note that $U_4$ is central in $U$, and that $U_3$ is central in $U$ if and only if the bilinear form $\ssT(\cdot,\cdot)$ is identically zero.

It will be convenient to divide Moufang hexagons into the following classes.
\medskip

\begin{compactenum}[($\mathrm{H}$1)]
\item The \textit{$\sG_2$-hexagons} associated to the Chevalley group $\sG_2(\FF)$ with $\kar\FF\neq 3$. These are associated to hexagonal systems $(\FF,\FF,\#)$ of type $1/\FF$ from \cite[(15.20)]{TW:02} with $\kar\FF\neq 3$, where $a^{\#}=a^2$ for all $a\in\FF$.
\item The \textit{$\sD_4$-hexagons} associated to Tits indices ${^3\mathsf{D_{4,2}^2}}$ and ${^6\mathsf{D_{4,2}^2}}$. These correspond to the hexagonal systems $(\EE,\FF,\#)$ of type $3/\FF$ from \cite[(15.21)]{TW:02} with $\EE/\FF$ a separable cubic extension (normal for ${^3}\sD_{4,2}^2$ and not normal for ${^6}\sD_{4,2}^2$), and $a^{\#}=a^{\sigma}a^{\sigma^2}$ with $\sigma$ a nontrivial element of $\mathrm{Gal}(\LL/\FF)$, with $\LL/\FF$ the normal closure of $\EE/\FF$. We shall abbreviate the notation ${^3}\sD_{4,2}^2$ to ${^3}\sD_4$ and ${^6}\sD_{4,2}^2$ to ${^6}\sD_4$.
\item The type \textit{$\sE$-hexagons} associated to Tits indices ${^1\mathsf{E_{6,2}^{16}}}$, ${^2\mathsf{E_{6,2}^{16''}}}$, and $\mathsf{E_{8,2}^{78}}$. These correspond to the hexagonal systems of type $9/\FF$ (for ${^1\mathsf{E_{6,2}^{16}}}$), $9\KK/\FF$ (for ${^2\mathsf{E_{6,2}^{16''}}}$), $27/\FF$, and $27\KK/\FF$ from \cite[(15.22), (15.29), (15.31), (15.34)]{TW:02}. We will recall some basic properties of hexagonal systems of type $9/\FF$ and $9\KK/\FF$ in Section~\ref{sec:ovoids}. It will turn out that we do not require any detailed information on systems of type $27/\FF$ or $27\KK/\FF$.
\item The \textit{mixed hexagons} associated to the hexagonal systems $(\EE,\FF,\#)$ where $\kar\FF=3$ and either $\EE=\FF$ or $\EE/\FF$ is a (necessarily purely inseparable) field extension with $\EE^3\subseteq \FF\subseteq\EE$, where $a^{\#}=a^2$ for all $a\in\EE$. These correspond to the hexagonal systems of type $1/\FF$ from \cite[(15.20)]{TW:02} with $\kar\FF=3$ (note that this class includes hexagons associated to the Chevalley group $\sG_2(\FF)$ with $\kar\FF=3$).
\end{compactenum}
\medskip

\noindent From the classification, if $(\JJ,\FF,\#)$ is an hexagonal system not in class (H4) then $\dim\JJ\in\{1,3,9,27\}$. We note that in classes (H1), (H2) and (H4) the vector space $\JJ$ has the additional algebraic structure of a field. Moreover, in an hexagonal system $(\JJ,\FF,\#)$ of type $9/\FF$ the vector space $\JJ$ has the structure of a (noncommutative) cyclic division algebra of degree three with centre~$\FF$. 

We shall fix the duality class of a Moufang hexagon throughout this paper as follows.

\begin{convention}\label{conv:duality}
 If $\Gamma$ is associated to the hexagonal system~$(\JJ,\FF,\#)$ then the points on a line are indexed by $\{\infty\}\cup \FF$, and the lines through a point are indexed by~$\{\infty\}\cup\JJ$. 
 \end{convention}

Thus, in particular, class (H1) consists precisely of the dual split Cayley hexagons over fields of characteristic different from~$3$, while the dual split Cayley hexagons over fields of characteristic~$3$ are a subset of class~(H4). We refer to the root subgroups $U_i$ with $i$ odd (respectively $i$ even) as the short (respectively long) root subgroups. The short (respectively long) \textit{root elations} are the conjugates of elements of $U_i$ with $i$ odd (respectively $i$ even).

For $a\in\JJ\backslash\{0\}$ let
$$
a^{-1}=\ssN(a)^{-1}a^{\#}.
$$
Following \cite[(32.12)]{TW:02}, for $a\in\JJ\backslash\{0\}$ and $t\in\FF\backslash\{0\}$ let
\begin{align*}
s_1(a)=x_7(a^{-1})x_1(a)x_7(a^{-1})\quad\text{and}\quad 
s_6(t)=x_{12}(t^{-1})x_6(t)x_{12}(t^{-1}),
\end{align*}
and so $s_1(a)\mapsto s_1$ and $s_6(t)\mapsto s_6$ under the homomorphism $N\to N/H=W$. In particular $Hs_1(a)=s_1(a)H=s_1(1)H$ and $Hs_6(t)=s_6(t)H=s_6(1)H$ for all $a\in\JJ\backslash\{0\}$ and $t\in\FF\backslash\{0\}$. We will sometimes write $s_i$ in place of $s_i(1)$, $i=1,6$, when there is no risk of ambiguity, however as we show below the elements $s_i(1)$ are not involutions (unless $\kar\FF=2$), and instead they have order~$4$.

Writing $g^h=hgh^{-1}$ we have (see \cite[(29.35), (32.12)]{TW:02}; note that our convention for $g^h$ differs from \cite{TW:02}, and is chosen as we often need to move a $s_i(\cdot)$ past an $x_j(\cdot)$ term from left to right -- thus the formulae from \cite{TW:02} have been modified accordingly):
\begin{align*}
x_1(b)^{s_1(a)}&=x_7(\ssT(a^{-1},b)a^{-1}-(a^{-1})^{\#}\times b)&x_1(a)^{s_6(t)}&=x_5(ta)\\
x_2(t)^{s_1(a)}&=x_6(t/\ssN(a))&x_2(u)^{s_6(t)}&=x_4(-tu)\\
x_3(b)^{s_1(a)}&=x_5(-a^{-1}\times b+\ssN(a)^{-1}\ssT(a,b)a)&x_3(a)^{s_6(t)}&=x_3(a)\\
x_4(t)^{s_1(a)}&=x_4(t)&x_4(u)^{s_6(t)}&=x_2(t^{-1}u)\\
x_5(b)^{s_1(a)}&=x_3(a\times b-\ssN(a)\ssT(a^{-1},b)a^{-1})&x_5(a)^{s_6(t)}&=x_1(-t^{-1}a)\\
x_6(t)^{s_1(a)}&=x_2(-t\ssN(a))&x_6(u)^{s_6(t)}&=x_{12}(t^{-2}u).
\end{align*}
In particular, note that $
s_i(1)x_j(a)s_i(1)^{-1}=x_{2i+6-j}(\epsilon_{ij}a)$ for $i\in\{1,6\}$ and $j\in\{1,2,3,4,5,6\}$ (with $a\in\JJ$ for $j$ odd and $a\in\FF$ for $j$ even), where $\epsilon_{ij}\in\{-1,1\}$ and the index $2i+6-j$ is read cyclically to lie between $1$ and~$12$. It is convenient to record the signs $\epsilon_{ij}$ in the following table (with $i$ indexing rows, and $j$ indexing columns):
\begin{align}\label{eq:signs}
\begin{array}{|c||c|c|c|c|c|c|}
\hline
\epsilon_{ij}& 1  & 2 & 3&4&5&6 \\
\hline\hline
1&1&1&1&1&-1&-1\\
\hline
6&1&-1&1&1&-1&1\\
\hline
\end{array}
\end{align}
It follows from~(\ref{eq:signs}) that $s_i^4x_j(a)s_i^{-4}=x_j(a)$ for $i\in\{1,6\}$ and $1\leq j\leq 6$, and so $s_1^4=s_6^4=1$. For example,
$$
s_1^4x_6(t)s_1^{-4}=s_1^3x_2(-t)s_1^{-3}=s_1^2x_6(-t)s_1^{-2}=s_1x_2(t)s_1^{-1}=x_6(t)
$$
(however note that $s_1^2x_6(t)s_1^{-2}=x_6(-t)$ and so $s_1^2\neq 1$ unless $\kar\FF=2$). 

We record some further formulae for later use. Using the definition of $s_1(a)$ and $s_6(t)$ we have (for $a\in\JJ$ and $t\in\FF$)
\begin{align}
\label{eq:1-fold}x_1(a)&=x_7(-a^{-1})s_1(a)x_7(-a^{-1})\\
\label{eq:6-fold}x_6(t)&=x_{12}(-t^{-1})s_6(t)x_{12}(-t^{-1}).
\end{align}
By \cite[(32.12)]{TW:02} we have
\begin{align}\label{eq:negativert}
\begin{aligned}
{[x_2(t),x_7(a)]}&=x_3(ta)x_4(-t^2\ssN(a))x_5(ta^{\#})x_6(-t\ssN(a))\\
[x_3(b),x_7(a)]&=x_4(-\ssT(a,b^{\#}))x_5(a\times b)x_6(-\ssT(a^{\#},b))\\
[x_5(b),x_7(a)]&=x_6(-\ssT(a,b))\\
[x_{12}(t),x_4(u)]&=x_2(tu)\\
[x_{12}(t),x_5(a)]&=x_1(-ta)x_2(-t^2\ssN(a))x_3(-ta^{\#})x_4(-t\ssN(a))
\end{aligned}
\end{align}
and $[x_i(c),x_7(a)]=1$ for $i\in \{4,6\}$, and $[x_{12}(t),x_i(c)]=1$ for $i\in\{1,2,3\}$. 

\subsection{Parabolic subgroups and coordinatisation}\label{sec:coordinates}

By Convention~\ref{conv:duality}, in the Dynkin diagram 
$\begin{tikzpicture}[scale=0.5,baseline=-0.5ex]
\node at (-1,0) {1};
\node at (1,0) {6};
\node [inner sep=0.8pt,outer sep=0.8pt] at (-0.5,0) (2) {$\bullet$};
\node [inner sep=0.8pt,outer sep=0.8pt] at (0.5,0) (3) {$\bullet$};
\phantom{\draw [line width=0.5pt,line cap=round,rounded corners] (2.north west)  rectangle (2.south east);}
\phantom{\draw [line width=0.5pt,line cap=round,rounded corners] (3.north west)  rectangle (3.south east);}
\draw (-0.5,0)--(0.5,0);
\draw (-0.5,0.11)--(0.5,0.11);
\draw (-0.5,-0.11)--(0.5,-0.11);
\draw (0+0.15,0.3) -- (0-0.08,0) -- (0+0.15,-0.3);
\end{tikzpicture}$ the points of $\Gamma$ are the type $6$ objects, and the lines of $\Gamma$ are the type $1$ objects (here $S=\{1,6\}$). Thus the points of $\Gamma$ are in bijection with the cosets $G/P_1$, where $P_1$ is the parabolic subgroup $P_1=B\cup Bs_1B$ (note: in the general building setup, the vertices of type $j$ correspond to the cosets of the parabolic subgroup $P_{S\backslash\{j\}}$, and in this case $S\backslash\{6\}=\{1\}$). Similarly the lines are in bijection with the cosets $G/P_6$ with $P_6=B\cup Bs_6B$. Thus the points of $\Gamma$ are 
\begin{center}
$P_1$,\quad $x_6(t)s_6P_1$,\quad $x_1(a)x_2(t)s_1s_6P_1$,\quad $x_6(t)x_5(a)x_4(t')s_6s_1s_6P_1$,\\
$x_1(a)x_2(t)x_3(a')x_4(t')s_1s_6s_1s_6P_1$,\quad $x_6(t)x_5(a)x_4(t')x_3(a')x_2(t'')s_6s_1s_6s_1s_6P_1$
\end{center}
with $t,t',t''\in \mathbb{F}$ and $a,a'\in\mathbb{J}$, and analogously for lines. 

Points $gP_1$ and $hP_1$ are at distance $0,2,4,6$ (in the incidence graph) if and only if
\begin{align}\label{eq:measuredistance}
P_1g^{-1}hP_1=P_1,\,P_1s_6P_1,\,P_1s_{616}P_1,\,P_1s_{61616}P_1,
\end{align}
respectively (where, for example, $s_{616}=s_6s_1s_6$). Dual statements apply for lines. The point $gP_1$ and the line $hP_6$ are incident if and only if $gP_1\cap hP_6\neq\emptyset$.

Abbreviating notation in the obvious way, the points of $\Gamma$ are given by all $n$-tuples ($1\leq n\leq 5$) in the sets $\FF$, $\JJ\times \FF$, $\FF\times \JJ\times \FF$, $\JJ\times \FF\times\JJ\times \FF$, $\FF\times \JJ\times \FF\times\JJ\times \FF$ together with a point labelled $(\infty)$ (corresponding to $P_1$). The lines are given by the $n$-tuples $(1\leq n\leq 5)$ in the sets $\JJ$, $\FF\times \JJ$, $\JJ\times \FF\times \JJ$, $\FF\times \JJ\times\FF\times \JJ$, $\JJ\times \FF\times \JJ\times\FF\times \JJ$, denoted with square brackets to distinguish from points, together with a line labelled $[\infty]$ (corresponding to $P_6$). This notation, along with the equations below determining the incidence relation, is called a \textit{coordinatisation} of $\Gamma$. The split Cayley hexagons and the triality hexagons are coordinatised in \cite[Chapter 3]{HVM:98}, and we extend this coordinatisation to general Moufang hexagons in the following theorem.

\begin{thm}\label{thm:coordinates}
Let $t,t',t'',u,u'\in\FF$ and $a,a',b,b',b''\in\JJ$. The incidence relation $*$ between points and lines is given by
\begin{align*}
&(t,a,t',a',t'')*[t,a,t',a']*(t,a,t')*[t,a]*(t)*[\infty]*\\
&(\infty)*[b]*(b,u)*[b,u,b']*(b,u,b',u')*[b,u,b',u',b''],
\end{align*}
and $(t,a,t',a',t'')*[b,u,b',u',b'']$ if and only if 
\begin{align}
\label{eq:system11}u&=t''+t\ssN(b)-\ssT(a',b)+\ssT(a,b^{\#})\\
\label{eq:system12}b'&=a'-(a\times b)-tb^{\#}\\
\label{eq:system13}u'&=t'+t^2\ssN(b)-tt''+t\ssT(a,b^{\#})+\ssT(a^{\#},b)-\ssT(a,a')\\
\label{eq:system14}b''&=a+tb.
\end{align}
if and only if 
\begin{align}
\label{eq:system21}a&=b''-tb\\
\label{eq:system22}t'&=u'+t^2\ssN(b)+ut-t\ssT(b'',b^{\#})+\ssT(b',b'')+\ssT(b''^{\#},b)\\
\label{eq:system23}a'&=b'+b\times b''-tb^{\#}\\
\label{eq:system24}t''&=u-t\ssN(b)+\ssT(b,b')+\ssT(b'',b^{\#}).
\end{align}
\end{thm}

\begin{proof}
All relations are clear with the exception of the relation $(t,a,t',a',t'')*[b,u,b',u',b'']$. Writing $g_1=x_6(t)x_5(a)x_4(t')x_3(a')x_2(t'')$ and $g_2=x_1(b)x_2(u)x_3(b')x_4(u')x_5(b'')$, the point $(t,a,t',a',t'')=g_1w_0P_1$ is on the line $[b,u,b',u',b'']=g_2w_0P_6$ if and only if 
$$
w_0^{-1}g_2^{-1}g_1w_0P_1\cap P_6\neq\emptyset.
$$
Thus if $g_2^{-1}g_1=u_6u_5u_4u_3u_2u_1$ with $u_i\in U_i$ it follows that $(t,a,t',a',t'')*[b,u,b',u',b'']$ if and only if $u_5=u_4=u_3=u_2=1$. 

We have $g_2^{-1}g_1=x_5(-b'')x_4(-u')x_3(-b')x_2(-u)x_1(-b)x_6(t)x_5(a)x_4(t')x_3(a')x_2(t'')$, and we use the commutator relations to write this element in $U_6U_5U_4U_3U_2U_1$ form. Noting that $U_4$ is central in $U$, we shall, for convenience, move all $U_4$ terms temporarily to the far right during the working. Moving the $x_6(t)$ term to the left requires commutator relations~(\ref{eq:commutator5}) and (\ref{eq:commutator4}) to move the term past $x_1(-b)$ and $x_2(-u)$, and we obtain 
\begin{align*}
g_2^{-1}g_1&=x_6(t)x_5(-b'')x_3(-b')x_2(-u)x_1(-b)x_2(t\ssN(b))x_3(tb^{\#})x_5(a+tb)x_3(a')x_2(t'')x_4(z_1)
\end{align*}
with $z_1=t'-u'-t^2\ssN(b)-tu$. Now, since $U_2$ commutes with all root subgroups $U_i$ with $i\neq 6$, we shall temporarily move all $U_2$ terms to the right, and record them next to the $U_4$ term, giving
\begin{align*}
g_2^{-1}g_1&=x_6(t)x_5(-b'')x_3(-b')x_1(-b)x_3(tb^{\#})x_5(a+tb)x_3(a')x_2(y_1)x_4(z_1)
\end{align*}
where $y_1=t''-u+t\ssN(b)$. We now move the term $x_5(a+tb)$ to the left. Invoking commutator relation~(\ref{eq:commutator2}), followed by~(\ref{eq:commutator3}), and then~(\ref{eq:commutator2}) again, we obtain
\begin{align*}
g_2^{-1}g_1&=x_6(t)x_5(-b'')x_3(-b')x_1(-b)x_5(a+tb)x_3(tb^{\#})x_3(a')x_2(y_1)x_4(z_2)\\
&=x_6(t)x_5(-b'')x_3(-b')x_5(a+tb)x_1(-b)x_3(-b\times (a+tb))x_3(tb^{\#}+a')x_2(y_2)x_4(z_3)\\
&=x_6(t)x_5(a-b''+tb)x_3(-b')x_1(-b)x_3(tb^{\#}+a'-a\times b-t(b\times b))x_2(y_2)x_4(z_4)
\end{align*}
where $z_2=z_1+t\ssT(a,b^{\#})+t^2\ssT(b,b^{\#})$, $z_3=z_2-\ssT(b,(a+tb)^{\#})$, $y_2=y_1-\ssT(b^{\#},a+tb)$, and $z_4=z_3+\ssT(a+tb,-b')$. Since $b\times b=2b^{\#}$ the second coefficient of $x_3(\cdot)$ simplifies to $a'-a\times b-tb^{\#}$, and we then move the $x_3(a'-a\times b-tb^{\#})$ term past $x_1(-b)$, using~(\ref{eq:commutator1}), giving
\begin{align*}
g_2^{-1}g_1&=x_6(t)x_5(a-b''+tb)x_3(-b')x_3(a'-a\times b-tb^{\#})x_1(-b)x_2(y_3)x_4(z_4)\\
&=x_6(t)x_5(a-b''+tb)x_4(z_4)x_3(a'-b'-a\times b-tb^{\#})x_2(y_3)x_1(-b)
\end{align*}
with $y_3=y_2-\ssT(b,a'-a\times b-tb^{\#})$. This is now in $U_6U_5U_4U_3U_2U_1$ form, and after simplification we obtain $y_3=t''-u+t\ssN(b)+\ssT(a,b^{\#})-\ssT(a',b)$ and 
$$z_4=t'-u'-t^2\ssN(b)-tu-t\ssT(a,b^{\#})-\ssT(a^{\#},b)-\ssT(a,b')-t\ssT(b,b').
$$ 
Equations~(\ref{eq:system11}), (\ref{eq:system12}), and (\ref{eq:system14}) now follow (from the conditions $u_2=u_3=u_5=1$), and moreover $z_4=0$ (since $u_4=1$). Using~(\ref{eq:system11}) and~(\ref{eq:system12}) to eliminate $u$ and $b'$ from the equation $z_4=0$ yields~(\ref{eq:system13}). 

To derive the equations~(\ref{eq:system21})--(\ref{eq:system24}), reverse the order of the first set of equations and rearrange the expressions to give expressions for $a,t',a',t''$. Now substitute to find expressions for $a,t',a',t''$ in terms of $t,b,u',b',u',b''$. 
\end{proof}

Commutator relations are used extensively in this paper (in particular in Section~\ref{sec:ovoids}), and we shall often give less details than in the above proof. 

\section{Domestic collineations of Moufang Hexagons}\label{sec:bulk}

Recall from Theorem~\ref{thm:PTM3} that a collineation of a thick hexagon is point-domestic if and only if the fixed element structure of is either a ball of radius $3$ in the incidence graph centred at a line, a large full subhexagon, or an ovoid (and dually for line-domestic collineations). We consider each case in turn in Sections~\ref{sec:elations}, \ref{sec:subhex}, and~\ref{sec:ovoids}, culminating in the proof of Theorem~\ref{thm:main} in Section~\ref{sec:main}.

\subsection{Balls of radius three in the incidence graph}\label{sec:elations}

Let $\Gamma$ be a Moufang hexagon, with Convention~\ref{conv:duality} in force.

\begin{lemma}\label{lem:central}
A collineation $\theta$ of a Moufang hexagon fixes precisely a ball of radius $3$ in the incidence graph centred at a point if and only if $\theta$ is conjugate to $x_4(1)$. 
\end{lemma}

\begin{proof}
Suppose that $\theta$ fixes a ball of radius $3$ centred at a point. After conjugating, we may assume that the centre of the fixed ball is the point $P_1$ (in the notation of Section~\ref{sec:coordinates}). In the $BN$-pair language, the hypothesis of the lemma gives that $\theta$ fixes each chamber $gB$ with $g\in B\cup Bs_1B\cup Bs_6B\cup Bs_1s_6B\cup Bs_6s_1B\cup Bs_1s_6s_1B$. In particular, $\theta B=B$, giving $\theta\in B$. Thus $\theta=hu$ with $h\in H$ and $u\in U$. Write 
$
u=x_1(a)x_2(t)x_3(a')x_4(t')x_5(a'')x_6(t'')
$ with $a,a',a''\in\JJ$ and $t,t',t''\in \FF$. For each $z\in\JJ$ the chamber $x_1(z)s_1B$ is fixed, and so by commutator relations $x_1(z)s_1B=\theta x_1(z)s_1B=hx_1(a+z)s_1B$, giving $hx_1(a+z)h^{-1}=x_1(z)$ for all $z\in\JJ$. Taking $z=0$ gives $x_1(a)=h^{-1}x_1(0)h=1$ and so $a=0$, and then $hx_1(z)h^{-1}=x_1(z)$ for all $z\in\JJ$ gives $h\in C(U_1)$ (the centraliser of $U_1$). 

Similarly we have $x_6(u)s_6B=\theta x_6(u)s_6B=hx_6(t''+u)s_6B$ for all $u\in\FF$, and as above this implies that  $t''=0$ and that $h\in C(U_6)$. Since $h$ fixes a subhexagon (as it fixes an apartment) and $h\in C(U_1)\cap C(U_6)$ it follows from \cite[Corollary~1.8.5]{HVM:98} that $h$ is the identity. Thus $\theta=x_2(t)x_3(a')x_4(t')x_5(a'')$. Continuing in this way, since the chamber $x_1(0)x_2(u)s_1s_6B$ ($u\in\FF$) is fixed by $\theta$, we have $t=0$. Similarly since the chamber $x_6(0)x_5(z)s_6s_1B$ ($z\in\JJ$) is fixed by $\theta$ we have $a''=0$. So $\theta=x_3(a')x_4(t')$, and since the chamber $x_1(0)x_2(0)x_3(z)s_1s_6s_1B$ ($z\in\JJ$) is fixed we have $a'=0$. Thus $\theta=x_4(t')$, and this is conjugate to $x_4(1)$ by an element of~$H$.

Conversely, it is easy to check that $x_4(1)$ fixes precisely a ball of radius~$3$ in the incidence graph centred at a point.
 \end{proof}

\begin{lemma}\label{lem:axial}
There exists a collineation of a Moufang hexagon~$\Gamma$ fixing precisely a ball of radius~$3$ in the incidence graph centred at a line if and only if $\Gamma$ is in class $(\mathrm{H4})$, and in this case all such collineations are conjugate to $x_3(1)$.
\end{lemma}

\begin{proof}
After conjugating we may assume that $\theta$ fixes a ball of radius $3$ centred at the line $P_6$. Thus $\theta$ fixes each chamber $gB$ with $g\in B\cup Bs_1B\cup Bs_6B\cup Bs_1s_6B\cup Bs_6s_1B\cup Bs_6s_1s_6B$. As in Lemma~\ref{lem:central} we see that $\theta=x_3(a)x_4(t)$ for some $a\in\JJ$ and $t\in\FF$. Since the chamber $x_6(0)x_5(0)x_4(u)s_6s_1s_6B$ ($u\in\FF$) is fixed we see that $\theta=x_3(a)$. Moreover, since the chamber $x_6(0)x_5(z)x_4(0)s_6s_1s_6B$ $(z\in\JJ$) is fixed by $\theta$, the commutator relations give
\begin{align*}
x_5(z)s_6s_1s_6B&=x_3(a)x_5(z)s_6s_1s_6B=x_5(z)x_4(\ssT(a,z))s_6s_1s_6B.
\end{align*}
Thus $\ssT(a,z)=0$ for all $z\in\JJ$, and so $\ssT(\cdot,\cdot)$ is degenerate (as $a\neq 0$, otherwise $\theta=1$). By \cite[(30.5)]{TW:02} this forces $\Gamma$ to be in class~(H4). Then $\theta=x_3(a)$ with $\ssT(a,z)=0$ for all $z\in\JJ$. This element fixes the ball $B\cup Bs_1B\cup Bs_6B\cup Bs_1s_6B\cup Bs_6s_1B\cup Bs_6s_1s_6B$. By the commutator relations $\theta$ is central in $U$, and $\theta$ is conjugate to $x_3(1)$.
\end{proof}

Thus we have determined all (necessarily domestic) collineations fixing preciesly a ball of radius $3$ in the incidence graph.

\subsection{Large full or ideal subhexagons}\label{sec:subhex}

We now turn to the possibility of collineations fixing large full (or ideal) subhexagons. We first recall the definition of regularity in hexagons (c.f. \cite[\S1.9]{HVM:98}). Let $x,y\in\cP\cup\cL$ be opposite elements of the hexagon $\Gamma$ (so $x,y$ are either both points, or both lines). For $i=2,3$ the \textit{distance-$i$-trace} associated to $\{x,y\}$ is $\Gamma_i(x)\cap\Gamma_{6-i}(y)$ (where $\Gamma_i(x)$ denotes the set of objects at distance $i$ from $x$ in the incidence graph of~$\Gamma$). The element $x$ is called \textit{distance-$i$-regular} if distinct distance-$i$-traces $\Gamma_i(x)\cap \Gamma_{6-i}(y)$ and $\Gamma_i(x)\cap \Gamma_{6-i}(y')$ (with $y,y'$ opposite $x$) have at most one element in common. The element $x$ is \textit{regular} if it is distance-$i$-regular for $i=2,3$. If all points of $\Gamma$ are distance-$i$-regular (respectively, regular) then we say that $\Gamma$ is \textit{point-distance-$i$-regular} (respectively, \textit{point-regular}). Dually, if all lines of $\Gamma$ are distance-$i$-regular (respectively, regular) then we say that $\Gamma$ is \textit{line-distance-$i$-regular} (respectively, \textit{line-regular}).

We record the following facts. 
\begin{prop}\label{prop:regular}
Let $\Gamma$ be a Moufang hexagon, with Convention~\ref{conv:duality} in force.
\begin{compactenum}[$(1)$]
\item If $\Gamma$ is in class $(\mathrm{H}1)$, $(\mathrm{H}2)$ or $(\mathrm{H}3)$ then $\Gamma$ is line-regular but not point-regular.
\item If $\Gamma$ is in class $(\mathrm{H4})$ then $\Gamma$ is both point-regular and line-regular.
\end{compactenum}
\end{prop}

\begin{proof}
By \cite{Ron:80} (see also \cite[Theorem~6.3.2]{HVM:98}) in any Moufang hexagon either all points are regular, or all lines are regular (or both). Moreover, since the regularity property is preserved on restriction to subhexagons, it follows from \cite[Corollary~3.5.11]{HVM:98} that for the hexagons in classes (H1), (H2), and (H3) (with Convention~\ref{conv:duality} in force) the lines are regular and the points are not regular, hence (1). Then (2) follows from \cite[Corollary~5.5.15]{HVM:98}. 
\end{proof}

We now return to the possibility of collineations fixing large full (or ideal) subhexagons. We first consider the mixed hexagons, class~(H4).

\begin{lemma}\label{lem:mixedsubhex}
A full or ideal subhexagon of a mixed Moufang hexagon $\Gamma$ is never the fixed point structure of an automorphism of~$\Gamma$.
\end{lemma}

\begin{proof}
Let $\Gamma'$ be a full suhexagon of a mixed hexagon $\Gamma$ with hexagonal system $(\EE,\FF,\#)$. Then either $\Gamma'$ is nonthick, or it is isomorphic to a mixed hexagon with hexagonal system $(\EE',\FF,\#)$ for some field $\EE'\leq \EE$. We claim that if $\Gamma'$ is fixed by $\theta$ then $\theta$ is the identity. Clearly it suffices to prove the result for the case that $\Gamma'$ is nonthick (for this is a subhexagon of the thick subhexagons). Let $\theta$ be a collineation of $\Gamma$ pointwise fixing $\Gamma'$. Let $p$ be an arbitrary point of $\Gamma'$ and $x\perp p$ an arbitrary point of $\Gamma$ collinear to $p$ not belonging to $\Gamma'$.  Pick two points $z_1$ and $z_2$ opposite $p$ but not opposite $x$ and such that $z_1^{\pperp}\cap p^\perp\ne z_2^{\pperp}\cap p^\perp$ (it does not matter whether one reads this inside $\Gamma$ or $\Gamma'$). The point-distance-2 regularity of $\Gamma$ (see Proposition~\ref{prop:regular}) implies that $z_1^{\pperp}\cap p^\perp\cap z_2^{\pperp}\cap p^\perp=\{x\}$. On the other hand it also implies that $z_i^{\pperp}\cap p^\perp=(z_i^\theta)^{\pperp}\cap p^\perp$ for $i=1,2$ (since $\theta$ fixes the two points of $z_i^{\pperp}\cap p^\perp$ in $\Gamma'$). It follows that $x^\theta=x$ and so every point collinear to $p$ is fixed, implying that the fixed point structure of $\theta$ is a full and ideal subhexagon, and hence has to coincide with $\Gamma$ by \cite[Proposition~1.8.2]{HVM:98}.  Hence $\theta$ is the identity. The dual argument applies due to line-distance-2-regularity of $\Gamma$ (see Proposition~\ref{prop:regular}).
\end{proof}

Now we consider the situation where $\Gamma$ is a $\sD_4$-hexagon or a type $\sE$-hexagon (classes (H2) and (H3)). Let $\FF$ be the underlying field and $(\J,\FF,\#)$ the corresponding hexagonal system. Let $\Gamma'$ be a full subhexagon. Since every subhexagon of a Moufang hexagon is again Moufang (see, for example, \cite[Lemma~5.2.2]{HVM:98}) the hexagon $\Gamma'$ corresponds to an hexagonal system $(\J',\FF,\#)$, with $\J'$ a subspace of $\J$. Since $\J$ and $\J'$ are vector spaces over $\FF$, they need to have a different dimension if $\Gamma'\neq\Gamma$. We first determine the pairs $(\J,\J')$ for which $\Gamma'$ is large in~$\Gamma$.

\begin{prop}\label{largesubh}
Let $\Gamma$ be a Moufang hexagon in class $(\mathrm{H}2)$ or $(\mathrm{H}3)$ with hexagonal system $(\JJ,\FF,\#)$. Let $\Gamma'$ be a full subhexagon of $\Gamma$ and let $(\JJ',\FF,\#)$ be the associated hexagonal system with $\JJ'\subseteq\JJ$ (here we shall allow $\dim\JJ'=0$ if $\Gamma'$ is non-thick). Then $\Gamma'$ is large in $\Gamma$ if and only if $\dim\JJ=3$ and $\JJ'=\FF$.
\end{prop}

\begin{proof}
The dimension of $\Gamma$ as an algebraic variety over $\FF$ is equal to $\dim\Gamma=3+2\dim\J$. The dimension of $\Gamma'$ is $\dim\Gamma'=3+2\dim\J'$, whereas the dimension of a point perp equals $\dim p^\perp=1+\dim\J$. If $\Gamma'$ is large in $\Gamma$, then every point of $\Gamma$ is in the perp of some point of $\Gamma'$ and so $\dim\Gamma\leq\dim\Gamma'+\dim p^\perp$, yielding $\dim\J\leq 1+2\dim\J'$. Since the possible dimensions of $\J$ and $\J'$ are $0,1,3,9$ and $27$, and $\dim\J\geq 3$, this implies $(\dim\J,\dim\J')=(3,1)$. We now show that this condition is also sufficient.

We use the coordinatisation of $\Gamma$ given in Theorem~\ref{thm:coordinates}. For now we do not place any restrictions on $\JJ$ and $\JJ'$. The subhexagon $\Gamma'$ is given by restricting $\J$ to $\J'$ in each of the coordinates associated to the short roots. A generic point outside $\Gamma'$ is given by $(t,a,t',a',t'')$, but by applying appropriate long root elations inside $\Gamma'$ we can assume that $t=t'=t''=0$. Hence $\Gamma'$ is large in $\Gamma$ if and only if, for all $(a,a')\in(\JJ\times \JJ)\setminus(\JJ'\times\JJ')$, there is some point $q$ of $\Gamma'$ collinear to the point $p:=(0,a,0,a',0)$. 

If $a\in\J'$, then $(0,a,0)\in p^\perp\cap\Gamma'$, and so we may assume that $a\in\J\setminus\J'$. By equations~(\ref{eq:system11})--(\ref{eq:system14}) the lines incident with $p$ are the lines
$$
L(b)=[b,-\ssT(a',b)+\ssT(a,b^{\#}),a'-(a\times b),-\ssT(a^{\#},b)-\ssT(a,b'),a]
$$
with $b\in\JJ$. Thus we must find a point $q$ of $\Gamma'$ on such a line $L(b)$. There are two scenarios. 
\smallskip

\noindent\textit{Scenario 1:} Suppose there is $b\in\JJ'$ such that $a'-(a\times b)\in\JJ'$. Then we may take 
$$
q=(b,-\ssT(a',b)+\ssT(a,b^{\#}),a'-(a\times b),-\ssT(a^{\#},b)-\ssT(a,b')).
$$
The point $q$ is in $\Gamma'$ and lies on $L(b)$ as required.
\smallskip

\noindent\textit{Scenario 2:} Suppose there is $b\in\JJ'$ and $u\in\FF$ such that $a-ub\in\JJ'$ and $a'-ub^{\#}\in\JJ'$. Then by~(\ref{eq:system21})--(\ref{eq:system24}) the point
$$
q=(u,a-ub,u',a'-ub^{\#},u'')
$$
(for any $u',u''\in\FF$) is in $\Gamma'$ and lies on $L(b)$ as required.
\smallskip

Thus it suffices to show that in the case $\dim\JJ=3$ and $\JJ'=\FF$ at least one of the above scenarios always occurs (given any $a\in\JJ\backslash\JJ'$ and $a'\in\JJ$). Assume first that $1,a,a'$ are linearly dependent, and so $a'=\lambda a+\mu$ with $\lambda,\mu\in\FF$. In this case, taking $b=-\lambda$ we have
$$
a'-(a\times b)=\lambda a+\mu+a\times \lambda=\lambda a+\mu+\lambda (a^{\sigma}+a^{\sigma^2})=\lambda \ssT(a)+\mu\in\FF,
$$
and so we are in the first scenario.

Assume now that $1,a,a'$ are linearly independent. Since $\dim\JJ=3$ there exist $\alpha,\beta,\gamma\in\FF$ with $a^{\sigma}a^{\sigma^2}=\alpha a'+\beta a+\gamma$. We have $\alpha\neq 0$ (for otherwise $a^{\sigma}a^{\sigma^2}=\beta a+\gamma$ implies that $\beta a^2+\gamma a\in\FF$, and so $a$ lies in a quadratic extension of~$\FF$, a contradiction). Choose $b=\alpha^{-1}(a+\beta)$ and $u=\alpha$. Then $a-ub=-\beta\in\FF$, and we compute
\begin{align*}
a'-ub^{\#}&=-\alpha^{-1}\beta\ssT(a)-\alpha^{-1}\gamma-\alpha^{-1}\beta^2\in\FF,
\end{align*}
and so we are in scenario~$2$, completing the proof. 
\end{proof}
%

\begin{thm}\label{thm:subhex}
Let $\Gamma$ be a Moufang hexagon, with Convention~\ref{conv:duality} in force. 
\begin{compactenum}[$(1)$]
\item $\Gamma$ does not admit a domestic collineation pointwise fixing precisely a large ideal subhexagon.
\item $\Gamma$ admits a collineation $\theta$ pointwise fixing precisely a large full subhexagon if and only if $\Gamma$ is either:
\begin{compactenum}[$(a)$]
\item a dual split Cayley hexagon over a field $\FF$ with $\kar\FF\neq 3$ and $X^2+X+1$ reducible over~$\FF$. In this case there is a unique class of nontrivial collineations fixing a large full subhexagon, with each such collineation having order~$3$.
\item a triality hexagon of type ${^3}\sD_4$ in which case $\theta$ is induced by a nontrivial element of the Galois group, and hence $\theta$ has order~$3$.
\end{compactenum}
\end{compactenum}
\end{thm}

\begin{proof}
(1) Theorem~5.9.11 of  \cite{HVM:98} implies that the Moufang hexagons in Classes~(H2) and (H3) have no thick ideal subhexagons, and in particular no large ones. By \cite[Theorem~6.10]{PVM:21} no hexagon in class (H1) admits a domestic automorphism fixing a large ideal subhexagon, and Lemma~\ref{lem:mixedsubhex} eliminates the possibility for hexagons in class~(H4).

(2) The statements for the split Cayley hexagons follow from the classification in \cite[Theorem~6.10]{PVM:21}. Thus suppose that $\Gamma$ is not of class (H1). If $\Gamma$ admits a collineation pointwise fixing precisely a large full subhexagon then by Proposition~\ref{largesubh} $\Gamma$ belongs to the class~(H2). It is clear that if $\theta$ pointwise fixes a dual split Cayley subhexagon $\Gamma'$, then the action on a point row is essentially a field automorphism (since the Moufang set structure of the point row is preserved by~$\theta$, it induces an element of $\mathsf{P}\Gamma\mathsf{L}_2(\JJ)$, with $\JJ$ the cubic extension of $\FF$ in question; this follows from a result of Hua \cite{Hua:49}, see also \cite[Lemma 8.5.10]{HVM:98}). Since the only subhexagon strictly containing  $\Gamma'$ is $\Gamma$ (in view of the dimension of the corresponding Jordan algebra), $\theta$ as a whole is determined by the field automorphism. Since the Galois group is trivial in the ${^6}\sD_4$ case, $\Gamma$ must be of type ${^3}\sD_{4}$. Since the fixed element set of a nontrivial element $\theta$ of the Galois group is precisely a dual split Cayley hexagon, and since this subhexagon is large and full by Proposition~\ref{largesubh}, it follows from Theorem~\ref{thm:PTM3} that $\theta$ is domestic.
\end{proof}

\subsection{Ovoids and spreads}\label{sec:ovoids}

We now turn to the possibility of collineations fixing ovoids or spreads.

\begin{thm}\label{regnocol}
If $\Gamma$ is a Moufang hexagon with regular points, then no nontrivial collineation fixes an ovoid. 
\end{thm} 

\begin{proof}
Let $O$ be an ovoid of $\Gamma$ fixed by some collineation $\theta$. We show that $\theta$ is necessarily the identity. Indeed, choose $p,q\in O$, $p\neq q$. There is a unique non-thick ideal subhexagon $\Gamma'$ containing $p$ and $q$ (see \cite[Lemma~1.9.10]{HVM:98}). Let $p*L_i*r_i*M_i*s_i*K_i*q$, $i=1,2$, be two distinct paths joining $p$ with $q$. Then we select an arbitrary line $M_3$ through $r_1$, $M_3\notin\{L_1,M_1\}$. Let $s_3$ be the unique point of $\Gamma'$ on $M_3$ distinct from $r_1$. Let $r$ be the unique member of $O$ collinear to~$s_3$. Then $r$ does not belong to $\Gamma'$ since it would otherwise not be opposite $q$ (by non-thickness of $\Gamma'$). Hence $r^{\pperp}\cap p^\perp$ contains $r_1$ but not $r_2$. Since $\theta$ fixes $p,q,r$, it fixes $p^\perp\cap q^{\pperp}\cap r^{\pperp}=\{r_1\}$ (by regularity of points). Hence $\theta$ fixes $L_1$. It is now easy to see that there are points of $O$ at distance 3 from every line meeting $L$ implying that $\theta$ pointwise fixes a full and ideal subhexagon, and hence is the identity by \cite[Proposition~1.8.2]{HVM:98}.   
\end{proof}

\begin{cor}\label{cor:spread}
The hexagons in classes $(\mathrm{H1})$, $(\mathrm{H2})$ and $(\mathrm{H3})$ do not admit collineations whose fixed element set is a spread, and those in class~$(\mathrm{H4})$ do not admit collineations whose fixed element set is either an ovoid or a spread.
\end{cor}

\begin{proof}
This follows from Proposition~\ref{prop:regular} and Theorem~\ref{regnocol}.
\end{proof}

We must determine whether the hexagons in class (H1), (H2) and (H3) admit collineations fixing ovoids. In \cite[Theorem~6.10]{PVM:21} we proved that for class (H1) such an automorphism exists if and only if $X^2+X+1$ is irreducible over~$\FF$, and so it remains to consider the classes (H2) and (H3).

An automorphism of an hexagonal system $(\JJ,\FF,\#)$ is a vector space isomorphism $h:\JJ\to\JJ$ such that $h\#=\#h$ (see~\cite[(15.17)]{TW:02}). In particular, if $h$ is an hexagonal system automorphism then $(ta^{\#})^{h}=ta^{h\#}$, $\ssT(a^{h},b^{h})=\ssT(a,b)$, $\ssN(a^{h})=\ssN(a)$, and $(a\times b)^{h}=a^{h}\times b^{h}$ for all $t\in\FF$ and $a,b\in\JJ$. Each hexagonal system automorphism $h:\JJ\to\JJ$ may also be regarded as an automorphism $h\in G$ of the associated Moufang hexagon in the natural way. Since $h$ fixes the base apartment we have $h\in H$, and 
\begin{align}\label{eq:hexauto}
h x_i(c)h^{-1}=x_i(c^{h})
\end{align}
for all $1\leq i\leq 12$, where $c\in\JJ$ for odd $i$, and $c\in\FF$ for even~$i$. 

Theorems~\ref{thm:ovoids1} and~\ref{thm:allexamples} below give the main tools to complete our analysis of collineations fixing ovoids. In these theorems we will make repeated use of the formulae listed in Section~\ref{sec:Moufang}. 

\begin{thm}\label{thm:ovoids1}
Let $\Gamma$ be a Moufang hexagon with hexagonal system $(\JJ,\FF,\#)$. Suppose that there exists a nontrivial point-domestic collineation $\theta$ of $\Gamma$ fixing an ovoid~$O$. Then $\theta$ is conjugate to an element $\theta=hx_1(1)s_1$ with $h:\JJ\to\JJ$ an hexagonal system automorphism. Moreover
\begin{compactenum}[$(1)$]
\item $h^3=1$ and $\theta$ has order~$3$;
\item $\ssT(a)=a+a^h+a^{h^2}$ and $\ssT(a^{\#})=\ssT(a,a^h)$ for all $a\in\JJ$;
\item the equation $z= 1-z^{-h}$ has no solution $z\in\JJ\backslash\{0\}$.
\end{compactenum}
\end{thm}

\begin{proof} Let $P_1=B\cup Bs_1B$ be the parabolic subgroup as in Section~\ref{sec:coordinates}, and so $G/P_1$ is the point set of $\Gamma$. Up to conjugation, we may assume that the points $P_1$ and $w_0P_1$ are fixed by $\theta$ (as $|O|\geq 2$ and $G$ acts strongly transitively on $\Gamma$). Then $\theta P_1=P_1$ gives $\theta\in P_1$, and since $\theta\notin B$ (or else a chamber is fixed) we have $\theta=b_1s_1b_2$ with $b_1,b_2\in B$. Conjugating further, we may assume that $\theta=bs_1$. Then write $b=hux_1(c)$ with $u\in U_6U_5U_4U_3U_2$ and $h\in H$ with $c\in\mathbb{J}$. Conjugating by an element of $H$ we may assume that $c=0$ or $c=1$. Now the fact that $\theta$ fixes $w_0P_1$ gives
$
w_0^{-1}hux_1(c)s_1w_0\in P_1,
$
which in turn gives $w_0^{-1}uw_0\in P_1$. But $w_0^{-1}uw_0\in U_{-6}U_{-5}U_{-4}U_{-3}U_{-2}$, and it follows that $u=1$. Thus $\theta=hx_1(c)s_1$ for some $h\in H$ and $c\in\{0,1\}$.

We claim that $c=1$. For if $c=0$ then $\theta=hs_1$. Consider the chamber $gB=x_6(t)s_6s_1s_6B$ with $t\neq 0$. Then
\begin{align*}
Bg^{-1}\theta gB&=Bs_{616}x_6(-t)hs_1x_6(t)s_{616}B=Bs_{616}x_6(-t)x_2(t')s_{1616}B
\end{align*}
for some $t'\in\FF\backslash\{0\}$. Since $x_6(-t)x_2(t')=x_2(t')x_6(-t)x_4(tt')=x_2(t')x_4(tt')x_6(-t)$ we can move $x_2(t')$ to the left, and $x_6(-t)$ to the right, where they are each absorbed into $B$. Thus
\begin{align*}
Bg^{-1}\theta gB&=Bs_{616}x_4(tt')s_{1616}B=Bx_{12}(\pm tt')s_{16161}B=Bw_0B,
\end{align*}
and so $gB$ is mapped to an opposite chamber, a contradiction. Thus we have shown that, up to conjugation, $\theta=hx_1(1)s_1$ for some $h\in H$.

Define $h:\JJ\to\JJ$ by 
$$
hx_1(a)h^{-1}=x_1(a^h)\quad\text{for $a\in\JJ$}.
$$
We will show that $h:\JJ\to\JJ$ is an hexagonal system automorphism, and that~(\ref{eq:hexauto}) holds. First we show that for $i\in\{2,4,6\}$ and $t\in\FF$ we have $hx_i(t)h^{-1}=x_i(t)$. Consider the case $i=2$. For $t\in\FF$ write $h^{-1}x_2(t)h=x_2(t')$. Let $gP_1=x_2(t)s_{1616}P_1$. Then
\begin{align*}
P_1g^{-1}\theta gP_1&=P_1s_{6161}x_2(-t)hx_1(1)s_1x_2(t)s_{1616}P_1=P_1s_{6161}x_2(-t')x_1(1)x_6(t)s_{616}P_1.
\end{align*}
One now uses commutator relations to push elements of $U_5\cup U_6$ to the left (where they move past $s_{6161}$, remain positive, and are absorbed into $P_1$), and push elements of $U_1\cup U_2\cup U_3$ to the right (where they move past $s_{616}$, remain positive, and are absorbed into $P_1$). Using this strategy, a short calculation gives
\begin{align*}
P_1g^{-1}\theta gP_1&=P_1s_{6161}x_4(t(t-t'))s_{616}P_1.
\end{align*}
If $t(t-t')\neq 0$ then since $s_{616}^{-1}x_4(t(t-t'))s_{616}\in U_{12}^*$ we have $P_1g^{-1}\theta gP_1=P_1s_{61616}P_1$, and so the point $gP_1$ is mapped onto an opposite point, a contradiction. Thus $t'=t$ as required. Very similar calculations apply for the cases $i=4,6$ by considering the points $x_4(t)s_{616}P_1$ and $x_6(t)s_{616}P_1$, respectively.
%

We return to the proof that $h:\JJ\to\JJ$ is an hexagonal system automorphism. It is clear from the definition that $h:\JJ\to \JJ$ is bijective with $(a+b)^h=a^h+b^h$, and since $0^h=0$ (as $h$ fixes the base apartment) we have $(-a)^h=-a^h$. It follows from the commutator relation $[x_1(a),x_6(t)]$, and the fact that $hx_j(t)h^{-1}=x_j(t)$ for $j\in\{2,4,6\}$ and $t\in\FF$, that $hx_i(a)h^{-1}=x_i(a^h)$ for all $i\in\{1,3,5\}$ and $a\in\JJ$. Then by equations~(\ref{eq:negativert}) this extends to negative root groups too. Moreover,
\begin{align*}
x_2(-t\ssN(a))x_3((ta^{\#})^{h})x_4(t^2\ssN(a))x_5((-ta)^{h})&=[x_1(a),x_6(t)]^h\\
&=[x_1(a^h),x_6(t)]\\
&=x_2(-t\ssN(a))x_3(ta^{h\#})x_4(t^2\ssN(a^h))x_5(-ta^h)
\end{align*}
shows that $(ta)^{h}=ta^h$ for all $t\in\FF$ and $a\in\JJ$, and that $a^{\#h}=a^{h\#}$ for all $a\in\JJ$. Hence $h:\JJ\to\JJ$ is an hexagonal system automorphism and~(\ref{eq:hexauto}) holds.

Now, by assumption $\theta=hx_1(1)s_1$ fixes an ovoid~$O$. Since the point $(\infty)=P_1$ is fixed, all other points of $O$ are opposite the point $P_1$. The points opposite~$P_1$ are of the form $gP_1=x_6(t)x_5(a)x_4(u)x_3(b)x_2(v)s_{61616}P_1$ with $t,u,v\in\FF$ and $a,b\in\JJ$, and a direct calculation with commutator relations gives
\begin{align*}
\theta gP_1&=hx_1(1)s_1x_6(t)x_5(a)x_4(u)x_3(b)x_2(v)s_{61616}P_1\\
&=hx_1(1)x_2(-t)x_3(-a)x_4(u)x_5(b)x_6(v)s_{61616}P_1\\
&=x_6(v)x_5(b^h-v)x_4(\alpha)x_3(\beta)x_2(\gamma)s_{61616}P_1
\end{align*}
where $\alpha=u-tv+v^2-\ssT(a,b)+\ssT(b^{\#})-v\ssT(b)$, $\beta=\ssT(b)-v-b^h-a^h$, and $\gamma=-t+\ssT(b)-v-\ssT(a)$. It follows from these equations that $p=(t,a,u,b,v)$ is fixed by $\theta$ if and only if $v=t$, $a=b^h-t$, and 
\begin{align}
\label{eq:ex1}\ssT(b,b^h)&=\ssT(b^{\#})\\
\label{eq:ex2}\ssT(b)&=b+b^h+b^{h^2}\\
\label{eq:ex3}\ssT(b^h)&=\ssT(b).
\end{align}
Recall that an ovoid $O$ has the property that each point of $\Gamma$ is at distance at most~$2$ (in the incidence graph) from a point of the ovoid. Consider the points $(a_1,t_1,a_2,t_2)$ with $a_1,a_2\in\JJ$ and $t_1,t_2\in\FF$. These points are at distance $4$ from $P_1=(\infty)$, and hence must be at distance exactly $2$ from one of the above fixed points of $\theta$. Thus there is a line $[a_1,t_1,a_2,t_2,a_3]$ (with $a_3\in\JJ$) containing one of the above fixed points. The equations in Theorem~\ref{thm:coordinates} then imply that for each $a\in\JJ$ there must be a fixed point $(\cdot,a,\cdot,\cdot,\cdot)\in O$. This in turn implies that equations~(\ref{eq:ex1})--(\ref{eq:ex3}) hold for all $b\in\JJ$. Hence statement (2) of the theorem holds.

Since $\ssT(a)=a+a^h+a^{h^2}$ and $\ssT(a^h)=\ssT(a)$ for all $a\in\JJ$ we have
$
a+a^h+a^{h^2}=a^h+a^{h^2}+a^{h^3},
$
from which it follows that $h^3=1$ (as an hexagonal system automorphism). Thus $h^3x_i(c)h^{-3}=x_i(c)$ for all $i$ and $c$, and so $h^3=1$ (as an element of~$G$). Then
\begin{align*}
\theta^3&=h^3x_1(1)s_1x_1(1)s_1x_1(1)s_1=s_1x_7(1)x_1(1)x_7(1)s_1^2=s_1^4
\end{align*}
and so $\theta^3=1$ (as $s_1^4=1$, as noted in Section~\ref{sec:Moufang}), hence (1).

Finally, since $\theta$ fixes no lines it fixes no chambers. Consider the chamber $gB=x_1(z)s_1B$. We have
\begin{align*}
\theta gB=hx_1(1)s_1x_1(z)s_1B=hx_1(1)x_7(z)B=hx_1(1)x_1(-z^{-1})s_1B=x_1(1-z^{-h})s_1B.
\end{align*}
Thus the equation $z=1-z^{-h}$ has no solution in~$\JJ$, proving (3).
\end{proof}

The following theorem gives a converse to Theorem~\ref{thm:ovoids1}.

\begin{thm}\label{thm:allexamples}
Let $(\JJ,\FF,\#)$ be an hexagonal system with Moufang hexagon $\Gamma$. Suppose there exists an hexagonal system automorphism $h:\JJ\to\JJ$ of order $1$ or $3$ such that $\ssT(a)=a+a^h+a^{h^2}$ and $\ssT(a^{\#})=\ssT(a,a^h)$ for all $a\in\JJ$. Then the automorphism $\theta=hx_1(1)s_1$ of $\Gamma$ is point-domestic. 
\end{thm}

\begin{proof}
Assume first that $|\FF|>2$. By \cite[Lemma~4.1]{PVM:19b} it suffices to show that no point opposite the base point~$P_1$ is mapped onto an opposite point by~$\theta$. A generic such point $p=(t,a,t',a',t'')$ in the $BN$-pair language is
$$
gP_1=x_6(t)x_5(a)x_4(t')x_3(a')x_2(t'')w_0P_1\quad\text{with $t,t',t''\in\FF$ and $a,a'\in\JJ$}.
$$
By~(\ref{eq:measuredistance}) point-domesticity is equivalent to the statement that $P_1g^{-1}\theta gP_1\neq P_1s_{61616}P_1$. 

Note that the formula $\ssT(c)=c+c^h+c^{h^2}$ and the fact that $h$ has order $1$ or $3$ implies that $\ssT(c^h)=\ssT(c)$ for all $c\in\JJ$.  A lengthy but straightforward calculation with commutator relations, using the formulae $\ssT(c)=c+c^h+c^{h^2}$, $\ssT(c^h)=\ssT(c)$, and $\ssT(c^{\#})=\ssT(c,c^h)$, shows that
\begin{align*}
P_1g^{-1}\theta gP_1&=P_1w_0x_6(f)x_5(\gamma)x_4(0)x_3(\gamma^h)x_2(f+\ssT(\gamma))w_0P_1,
\end{align*}
where $f=t''-t$ and $\gamma=a'^h-a-t''$.

Suppose first that $f\neq 0$. By~(\ref{eq:6-fold}) we have
\begin{align*}
P_1g^{-1}\theta gP_1&=P_1w_0x_{12}(-f^{-1})s_6(f)x_{12}(-f^{-1})x_5(\gamma)x_4(0)x_3(\gamma^{h})x_2(f+\ssT(\gamma))w_0P_1\\
&=P_1w_0s_6(f)x_{12}(-f^{-1})x_5(\gamma)x_4(0)x_3(\gamma^{h})x_2(f+\ssT(\gamma))w_0P_1.
\end{align*}
Note that $P_1w_0s_6(f)=P_1w_0s_6$ (as $H\leq P_1$). Using the formulae in~(\ref{eq:negativert}) to push the $x_{12}(-f^{-1})$ term to the right (where it is absorbed into $P_1$), we obtain (after some calculation)
\begin{align*}
P_1g^{-1}\theta gP_1&=P_1w_0s_6x_5(\gamma)x_4(f^{-1}\ssN(\gamma))
x_3(f^{-1}\gamma^{\#}+\gamma^{h})x_2(f_1)w_0P_1\\
&=P_1w_0x_1(-\gamma)x_2(f^{-1}\ssN(\gamma))
x_3(f^{-1}\gamma^{\#}+\gamma^{h})x_4(-f_1)s_{1616}P_1,
\end{align*}
where $f_1=f+\ssT(\gamma)+f^{-2}\ssN(\gamma)+f^{-1}\ssT(\gamma^{\#})$. 

Now suppose further that $\gamma\neq 0$. Similarly to the above, by~(\ref{eq:1-fold}) we have
\begin{align*}
P_1g^{-1}\theta gP_1&=P_1w_0s_1x_7(\gamma^{-1})x_2(f^{-1}\ssN(\gamma))
x_3(f^{-1}\gamma^{\#}+\gamma^{h})x_4(-f_1)s_{1616}P_1.
\end{align*}
Using the formulae in~(\ref{eq:negativert}) to push the $x_7(\gamma^{-1})$ term to the right (where it is absorbed into $P_1$) we obtain
\begin{align*}
P_1g^{-1}\theta gP_1&=P_1w_0s_1x_2(f^{-1}\ssN(\gamma))x_3(\gamma^{h})x_4(-f)s_{1616}P_1\\
&=P_1w_0x_6(f^{-1}\ssN(\gamma))x_5(\gamma^{h})x_4(-f)s_{616}P_1.
\end{align*}
Since $f^{-1}\ssN(\gamma)\neq 0$ we can use~(\ref{eq:6-fold}) again, giving
\begin{align*}
P_1g^{-1}\theta gP_1&=P_1w_0s_6x_{12}(-f\ssN(\gamma)^{-1})x_5(\gamma^{h})x_4(-f)s_{616}P_1\\
&=P_1w_0s_6x_5(\gamma^{h})x_4(0)s_{616}P_1\\
&=P_1w_0x_1(-\gamma^{h})s_{16}P_1,
\end{align*}
and using~(\ref{eq:1-fold}) gives
\begin{align*}
P_1g^{-1}\theta gP_1&=P_1w_0s_1x_7(\gamma^{-h})s_{16}P_1=P_1w_0s_{6}P_1=P_1s_{616}P_1,
\end{align*}
showing that the point $gP_1$ is mapped by $\theta$ to distance $4$ from $gP_1$ (see~(\ref{eq:measuredistance})). 

Simpler calculations show that 
\begin{align*}
P_1g^{-1}\theta gP_1&=\begin{cases}
P_1&\text{if $f=\gamma=0$}\\
P_1s_{616}P_1&\text{if either $f=0$ and $\gamma\neq 0$, of $f\neq 0$ and $\gamma=0$,}
\end{cases}
\end{align*}
completing the proof for the case $|\FF|>2$.

If $|\FF|=2$ then $\Gamma$ is either the dual split Cayley hexagon with parameters $(2,2)$, or is the triality hexagon with parameters $(2,8)$. In the first case $h$ is necessarily trivial, and in the second case $h$ is a nontrivial element of the Galois group of the cubic extension $\FF_8/\FF_2$. The results are easily verified in these cases. 
\end{proof}

\begin{cor}\label{cor:noovoidsH2}
No nontrivial collineation of a Moufang hexagon in class $(\mathrm{H2})$ fixes pointwise an ovoid. 
\end{cor}

\begin{proof}
Let $\Gamma$ be a Moufang hexagon in class (H2) with hexagonal system $(\EE,\FF,\#)$. Thus $\EE$ is a separable cubic field extension of $\FF$ (normal in the case ${^3}\sD_4$, and not normal in the case ${^6}\sD_4$). Suppose that $\theta$ is a nontrivial collineation pointwise fixing an ovoid. Then by Theorem~\ref{thm:ovoids1} we have, up to conjugation, $\theta=hx_1(1)s_1$ where $h:\EE\to\EE$ is an hexagonal system automorphism of order $1$ or $3$ with $\ssT(a)=a+a^h+a^{h^2}$ and $\ssT(a^{\#})=\ssT(a,a^h)$ for all $a\in\EE$. We have $h\neq 1$ (for otherwise $\ssT(a)=3a$ for all $a\in\EE$), and so $h$ has order~$3$.

By~\cite[p.2]{Jac:68} the vector space automorphism $h:\EE\to\EE$ is in fact an element of $\mathrm{Gal}(\EE/\FF)$. Since $h$ has order $3$ the ${^6}\sD_4$ case is eliminated (as the Galois group is trivial in this case). Thus $\Gamma$ is of type ${^3}\sD_4$. Choose any $a\in\JJ$ with $a^h\neq a$ (such $a$ exists as $h\neq 1$), and let $b=a-a^{h}$. Then $b\neq 0$ and $\ssT(b)=\ssT(a)-\ssT(a^h)=0$. Let $z_0=-b^{h^2}b^{-1}$. Since $h$ is a field automorphism we have
$$
z_0-1+z_0^{-h}=-b^{h^2}b^{-1}-1-b^hb^{-1}=-(b^{h^2}+b+b^h)b^{-1}=-\ssT(b)b^{-1}=0,
$$
contradicting Theorem~\ref{thm:ovoids1}(3).
\end{proof}

We finally turn our attention to hexagons in class (H3). We first discuss the connection between hexagonal systems of type $9\KK/\FF$ and those of type $9/\KK$. Let $(\JJ,\FF,\#)$ be an hexagonal system of type $9\KK/\FF$. By \cite[(15.39), (15.41)]{TW:02} the hexagonal system $(\JJ,\FF,\#)$ embeds into an hexagonal system $(\JJ_{\KK},\KK,\#)$ of type $9/\KK$, where $\JJ_{\KK}=\JJ\otimes_{\FF}\KK$ and $\KK/\FF$ is a quadratic Galois extension. The extension is determined by first choosing any $\delta\in\KK\backslash\FF$, and then defining $\#$ on $\JJ_{\KK}$ by 
$$
(a+b\delta)^{\#}=a^{\#}+(a\times b)\delta+b^{\#}\delta^2\quad\text{for $a,b\in\JJ$}.
$$
Recall from \cite[(15.5),(15.22)]{TW:02} that $\JJ_{\KK}$ has the algebraic structure of a cyclic division algebra of degree three with centre~$\KK$.

\begin{lemma}\label{lem:extensionh}
Let $(\JJ,\FF,\#)$ be an hexagonal system of type $9\KK/\FF$ and let $(\JJ_{\KK},\KK,\#)$ be the associated hexagonal system of type $9/\KK$ described above. Suppose that $h:\JJ\to\JJ$ is an hexagonal system automorphism of order~$3$ with $\ssT(a)=a+a^h+a^{h^2}$ and $\ssT(a^{\#})=\ssT(a,a^h)$ for all $a\in\JJ$. Then $h$ extends to an order $3$ hexagonal system automorphism $h:\JJ_{\KK}\to\JJ_{\KK}$ with $\ssT(\alpha)=\alpha+\alpha^h+\alpha^{h^2}$ and $\ssT(\alpha^{\#})=\ssT(\alpha,\alpha^h)$ for all $\alpha\in\JJ_{\KK}$.
\end{lemma}

\begin{proof}
Define $(a+b\delta)^{h}=a^h+b^h\delta$ for $a,b\in\JJ$. Let $X^2-tX-s$ be the minimal polynomial of $\delta$ over~$\FF$. Then
\begin{align*}
(a+b\delta)^{\#h}&=((a^{\#}+sb^{\#})+(a\times b+tb^{\#})\delta)^{h}=(a^{\#}+sb^{\#})^h+(a\times b+tb^{\#})^h\delta,
\end{align*}
and since $h:\JJ\to\JJ$ is an hexagonal system automorphism it follows that 
$$
(a+b\delta)^{\#h}=(a^{h\#}+sb^{h\#})+(a^h\times b^h+tb^{h\#})\delta=(a+b\delta)^{h\#}.
$$
Thus $h:\JJ_{\KK}\to\JJ_{\KK}$ is an hexagonal system automorphism, and it is clear that it has order~$3$. 

Since $\ssT(\alpha)=\ssT(a)+\ssT(b)\delta$ if $\alpha=a+b\delta$ we have
$
\ssT(\alpha)=\alpha+\alpha^{h}+\alpha^{h^2}
$
for all $\alpha\in\JJ_{\KK}$. Moreover, since $\ssT(\alpha,\beta)=\ssT(\alpha\beta)$ in systems of type $9/\KK$ (see \cite[(15.6), (15.22)]{TW:02}) we have 
\begin{align*}
\ssT(\alpha,\alpha^h)&=\ssT(a+b\delta,a^h+b^h\delta)\\
&=\ssT(aa^h+(ab^h+ba^h)\delta+bb^h\delta^2)\\
&=\ssT(a,a^h)+\ssT(ab^h+ba^h)\delta+\ssT(b,b^h)\delta^2\\
&=\ssT(a^{\#})+[\ssT(a,b^h)+\ssT(b,a^h)]\delta+\ssT(b^{\#})\delta^2.
\end{align*}
But since $a\times b=(a+b)^{\#}-a^{\#}-b^{\#}$ (see \cite[(15.15)]{TW:02}), we have
\begin{align*}
\ssT(a\times b)&=\ssT((a+b)^{\#})-\ssT(a^{\#})-\ssT(b^{\#})\\
&=\ssT(a+b,a^h+b^h)-\ssT(a^{\#})-\ssT(b^{\#})\\
&=\ssT(a,a^h)+\ssT(a,b^h)+\ssT(b,a^h)+\ssT(b,b^h)-\ssT(a^{\#})-\ssT(b^{\#})\\
&=\ssT(a,b^h)+\ssT(b,a^h).
\end{align*}
Thus 
\begin{align*}
\ssT(\alpha,\alpha^h)&=\ssT(a^{\#})+\ssT(a\times b)\delta+\ssT(b^{\#})\delta^2=\ssT(a^{\#}+(a\times b)\delta+b^{\#}\delta^2)=\ssT(\alpha^{\#}),
\end{align*}
completing the proof.
\end{proof}

Recall from \cite[(15.5), (15.9), (15.22)]{TW:02} the explicit model of the $9/\FF$ hexagonal systems. In particular there is a cubic Galois extension $\EE/\FF$, a generator $\sigma\in\mathrm{Gal}(\EE/\FF)$, and an element $\gamma\in\FF\backslash\ssN(\EE)$ such that each element $\alpha\in\JJ$ can be written in a unique way as
$$
\alpha=a+by+cy^2\quad\text{with $a,b,c\in\EE$}
$$
with multiplication given by the rules
$$
y^3=\gamma\quad\text{and}\quad ya=a^{\sigma}y\quad\text{for all $a\in\EE$}.
$$
We shall record this situation by denoting the division algebra $\JJ$ by $\JJ=(\EE,\sigma,\gamma)$. We have
\begin{align}
\label{eq:adjoint}(a+by+cy^2)^{\#}&=(a^{\#}-\gamma b^{\sigma}c^{\sigma^2})+(\gamma c^{\sigma\#}-a^{\sigma^2}b)y+(b^{\sigma^2\#}-a^{\sigma}c)y^2.
\end{align}

\begin{lemma}\label{lem:basis}
Let $(\JJ,\FF,\#)$ be an hexagonal system of type $9/\FF$ with $\JJ=(\EE,\sigma,\gamma)$. Let $e\in\EE\backslash \FF$ and $z\in\JJ\backslash\EE$. The elements 
$
1,e,e^{\#},z,z^{\#},e\times z,e\times z^{\#},e^{\#}\times z,e^{\#}\times z^{\#}
$
form a basis of~$\JJ$.
\end{lemma}

\begin{proof}
Write $z=a+by+cy^2$. Since $\{1,e,e^{\#}\}$ span $\EE$ (because the span is closed under $\#$, has dimension at most $3$, and strictly contains $\FF$) it is clear that we may assume that $a=0$. Moreover, by \cite[(15.6)(iv) and (15.15)(xi)]{TW:02} we have $z^2=\ssT(z)z-1\times z^{\#}=\ssT(z)z-\ssT(z^{\#})+z^{\#}$ and so
$
z^{\#}=z^2-\ssT(z)z+\ssT(z^{\#}),
$
and so we may replace each occurrence of $z^{\#}$ in the spanning set by $z^2$ without changing the vector space spanned (recall that $u\times v=(u+v)^{\#}-u^{\#}-v^{\#}$). We compute $
z^2=\gamma(bc^{\sigma}+b^{\sigma}c)+\gamma cc^{\sigma^2}y+bb^{\sigma}y^2$, 
and thus we may further replace each occurrence of $z^2$ in the spanning set with 
$
z'=z^2-\gamma(bc^{\sigma}+b^{\sigma}c)=\gamma cc^{\sigma^2}y+bb^{\sigma}y^2$
without changing the space spanned. Thus we must show that $\JJ$ is spanned by the elements $1,e,e^{\#},z,z',e\times z,e\times z',e^{\#}\times z,e^{\#}\times z'$. Let $\mathbb{V}$ be the span of these elements. Since $z=\ssT(z)-1\times z$ we have 
$$
\mathbb{V}\ni \lambda z+\mu (e\times z)+\nu (e^{\#}\times z)=\lambda \ssT(z)+(-\lambda +\mu e+\nu e^{\#})\times z,
$$
and since $\{1,e,e^{\#}\}$ is an $\FF$-basis of $\EE$ it follows that $f\times z\in\mathbb{V}$, and similarly $f\times z'\in\mathbb{V}$, for all $f\in\EE$. For $f\in\EE$ we have
\begin{align*}
f\times z&=-f^{\sigma^2}by-f^{\sigma}cy^2\\
f\times z'&=-\gamma f^{\sigma^2}cc^{\sigma^2}y-f^{\sigma}bb^{\sigma}y^2.
\end{align*}
Thus it is clear that if either $b=0$ or $c=0$ then $gy,gy^2\in\mathbb{V}$ for all $g\in\EE$, and hence $\mathbb{V}=\JJ$. So suppose that $b,c\neq 0$. Let $\lambda=\ssN(b)-\gamma\ssN(c)$, and note that $\lambda\in\FF$ with $\lambda\neq 0$ (as $\gamma\notin\ssN(\EE)$). Let $g\in\EE$ be arbitrary, and let $g_1=-\lambda^{-1}bb^{\sigma^2}g^{\sigma}$ and $g_2=-\lambda^{-1}c^{\sigma^2}g^{\sigma}$. Then
\begin{align*}
\mathbb{V}&\ni g_1\times z\,=\lambda^{-1}\ssN(b)gy+\lambda^{-1}bb^{\sigma}cg^{\sigma^2}y^2\\
\mathbb{V}&\ni g_2\times z'=\lambda^{-1}\gamma \ssN(c)gy+\lambda^{-1}bb^{\sigma}cg^{\sigma^2}y^2.
\end{align*}
Subtracting gives that $\lambda^{-1}(\ssN(b)-\gamma \ssN(c))gy=gy\in\mathbb{V}$ for all $g\in\EE$. Similarly we have $gy^2\in\mathbb{V}$ for all $g\in\EE$, and hence $\mathbb{V}=\EE$ as required. 
\end{proof}

\begin{cor}\label{cor:noovoidsH3}
No nontrivial collineation of a Moufang hexagon in class $(\mathrm{H3})$ fixes pointwise an ovoid. 
\end{cor}

\begin{proof}
Let $\Gamma$ be a Moufang hexagon in class (H3) with hexagonal system $(\JJ,\FF,\#)$. Suppose that $\theta$ is a nontrivial collineation pointwise fixing an ovoid. Then by Theorem~\ref{thm:ovoids1} we have, up to conjugation, $\theta=hx_1(1)s_1$ where $h:\JJ\to\JJ$ is an hexagonal system automorphism of order $1$ or $3$ with $\ssT(a)=a+a^h+a^{h^2}$ and $\ssT(a^{\#})=\ssT(a,a^h)$ for all $a\in\JJ$. We have $h\neq 1$ (for otherwise $\ssT(a)=3a$ for all $a\in\JJ$) and so $h$ has order~$3$.

By \cite[(30.6)]{TW:02} there is $\JJ'\leq \JJ$ such that $(\JJ',\FF,\#)$ is of class (H2). Let $a\in\JJ'\backslash\FF$, and so $\JJ'=\mathrm{span}_{\FF}\{1,a,a^{\#}\}$. If $a^h\in\JJ'$ then $\JJ'$ is stable under $h$ (as $a^{h^2}=\ssT(a)-a-a^h$), and so $\theta$ stabilises the Moufang hexagon $\Gamma'$ associated to $(\JJ',\FF,\#)$ and is point domestic fixing no lines of this hexagon, contradicting Corollary~\ref{cor:noovoidsH2}. Thus $a^h\notin\JJ'$. Then by \cite[(30.17)]{TW:02} there is a subspace $\JJ''\leq \JJ$ with $\JJ'\cup\{a^h\}\subseteq \JJ''$ such that $(\JJ'',\FF,\#)$ is an hexagonal system of type $9/\FF$ or $9\KK/\FF$ for some quadratic Galois extension $\KK/\FF$. It follows from Lemma~\ref{lem:basis} (along with the discussion just before Lemma~\ref{lem:extensionh} to extend a $9\KK/\FF$ system to a $9/\KK$ system) that the elements $1,a,a^{\#},a^h,a^{h\#},a\times a^h,a\times a^{h\#},a^{\#}\times a^h,a^{\#}\times a^{h\#}$ form a basis of $\JJ''$. Hence $\JJ''$ is stable under $h$, and so the Moufang hexagon $\Gamma''$ associated to $(\JJ'',\FF,\#)$ is stabilised by $\theta$. Moreover, since $\theta$ is point-domestic and fixes no lines of $\Gamma$ then $\theta$ restricted to $\Gamma''$ also has these properties, and so $\theta$ fixes an ovoid of $\Gamma''$. 

Thus it is sufficient to eliminate the possibility of a nontrivial collineation $\theta$ of a Moufang hexagon of type $9/\FF$ or $9\KK/\FF$ fixing pointwise an ovoid. Consider first the $9/\FF$ case. Let $(\JJ,\FF,\#)$ be of type $9/\FF$, and so $\JJ$ is a noncommutative cyclic division algebra of degree three with centre~$\FF$. By~\cite[p.2]{Jac:68} the vector space automorphism $h:\JJ\to\JJ$ is either an algebra automorphism, or an algebra anti-automorphism. The latter case is impossible as $h$ has order~$3$, and hence $h$ is an algebra automorphism. Following the argument of Corollary~\ref{cor:noovoidsH2} there is $b\in\JJ\backslash\{0\}$ with $\ssT(b)=0$. Then $z_0=-b^{h^2}b^{-1}$ satisfies $z_0-1+z_0^{-h}=-\ssT(b)b^{-1}=0$, contradicting Theorem~\ref{thm:ovoids1}(3). 

We now consider the $9\KK/\FF$ case. We cannot directly use the argument of the previous paragraph, because hexagonal systems of type $9\KK/\FF$ lack the algebraic structure required to form the element $z_0=-b^{h^2}b^{-1}$. Instead we argue as follows. Let $\Gamma$ be a Moufang hexagon with hexagonal system $(\JJ,\FF,\#)$ of type $9\KK/\FF$. If $\Gamma$ admits a point-domestic collineation fixing an ovoid then by Theorem~\ref{thm:ovoids1} there is an hexagonal system automorphism $h:\JJ\to\JJ$ of order $3$ such that $\ssT(a)=a+a^h+a^{h^2}$ and $\ssT(a^{\#})=\ssT(a,a^h)$ for all $a\in\JJ$. By Lemma~\ref{lem:extensionh} the map $h$ extends to an automorphism of the hexagonal system $(\JJ_{\KK},\KK,\#)$ of type $9/\KK$ such that $h:\JJ_{\KK}\to\JJ_{\KK}$ has order~$3$ and satisfies $\ssT(\alpha)=\alpha+\alpha^h+\alpha^{h^2}$ and $\ssT(\alpha^{\#})=\ssT(\alpha,\alpha^h)$ for all $\alpha\in\JJ_{\KK}$. But then by Theorem~\ref{thm:allexamples} the extended automorphism $\theta=hx_1(1)s_1$ of the associated hexagon $\Gamma_{\KK}$ of type $9/\KK$ is point-domestic. Since $\theta$ does not fix an ovoid of $\Gamma_{\KK}$ (by the previous paragraph) it must fix a large full subhexagon of $\Gamma_{\KK}$ (by Theorem~\ref{thm:PTM3} and Lemma~\ref{lem:axial}), contradicting Proposition~\ref{largesubh}. 
\end{proof}
%

\subsection{Proof of the main theorem}\label{sec:main}

We now have all ingredients for the proof of Theorem~\ref{thm:main}.

\begin{proof}[Proof of Theorem~\ref{thm:main}]
Let $\Gamma$ be a Moufang hexagon (with Convention~\ref{conv:duality} in force). Suppose that $\theta$ is a nontrivial line-domestic collineation. By Theorem~\ref{thm:PTM3} the fixed element structure of~$\theta$ is either (i) a ball of radius $3$ centred at a point, (ii) a large ideal subhexagon, or (iii) a spread. Case (ii) is eliminated by Theorem~\ref{thm:subhex}, and case (iii) is eliminated by Corollary~\ref{cor:spread}. By Lemma~\ref{lem:central} there is a unique class of collineations in case (i).

Suppose now that $\theta$ is a nontrivial point-domestic collineation. By Theorem~\ref{thm:PTM3} the fixed element structure of~$\theta$ is either (i) a ball of radius $3$ centred at a line, (ii) a large full subhexagon, or (iii) an ovoid. Lemma~\ref{lem:axial} deals with case (i), and Theorem~\ref{thm:subhex} deals with case (ii). Case (iii) is dealt with by Corollaries~\ref{cor:spread}, \ref{cor:noovoidsH2} and~\ref{cor:noovoidsH3}.

Finally, the statements on exceptional domestic collineations follow from Theorems~\ref{thm:TMPex1} and~\ref{thm:TMPex2}.
\end{proof}

\subsection{Concluding comments}\label{sec:extra1}

We conclude by providing an independent geometric proof of Corollary~\ref{cor:noovoidsH2}, and a uniform description of all examples of point-domestic collineations of Moufang hexagons. 

Recall from \cite{Ron:80} (see also \cite[Theorem~6.3.2]{HVM:98}) that all points and all lines of a Moufang hexagon $\Gamma$ are distance-$3$-regular. Suppose that $p,q\in\Gamma$ are opposite points, and consider the set $L(p,q)$ of all lines that are at distance $3$ in the incidence graph from both $p$ and $q$. The \textit{imaginary line} determined by $p,q$ is 
$$
I(p,q)=\{r\in\cP\mid d(L,r)=3\text{ for all $L\in L(p,q)$}\},
$$
where $d(L,r)$ denotes distance in the incidence graph. By distance-$3$-regularity the set $I(p,q)$ is determined by any two lines $L,L'\in L(p,q)$ with $L\neq L'$. 

\begin{prop}\label{ovoidsgen}
Let $\Gamma$ be a Moufang hexagon and suppose $\theta$ is a domestic collineation only fixing points (hence the fixed point structure is an ovoid~$O$). Then 
\begin{compactenum}[$(i)$] 
\item $\theta$ has order $3$, 
\item there exists a full dual split Cayley subhexagon stabilised by $\theta$, 
\item every full subhexagon stabilised under $\theta$ contains an ovoid fixed by $\theta$, \item $O$ is closed under taking imaginary lines.
\end{compactenum}
\end{prop}
\begin{proof}
Let $p,q$ be two points of the ovoid $O$ (the fixed point set of $\theta$). Let $p*L_i*r_i*M_i*s_i*K_i*q$, $i=1,2$, be two distinct paths joining $p$ with $q$, with $L_1^\theta=L_2$. Select $r_1'\in L_1\setminus\{p,r_1\}$ and let $s_1'$ be the unique point collinear to $r_1'$ and at distance 3 from $K_2$. Then $s_1'$ is contained in the unique full nonthick subhexagon $\Gamma''$ defined by $L_1$ and $K_2$. Let $r\in O$ be collinear to $s_1'$. Then $r\notin\Gamma''$ since $r$ is opposite $p$ and $q$. So $rs_1'$ does not belong to $\Gamma''$ and hence the unique line $M_2'$ through $r_2$ at distance $4$ from $rs_1'$ is distinct from $M_2$ and from $L_2$. Hence $r_2,s_1$, $M_2,M_2',L_2$ are contained in a unique full dual split Cayley subhexagon $\Gamma'$. The latter contains $r_1'$ hence $s_1'$ hence $rs_1'$ hence $r$. Now the lines $L_2,M_2$ and $K_2$ are contained in $\Gamma'\cap\Gamma'^\theta$, and so is the point~$r$. It follows that the shortest path from $r$ to $K_2$ is contained in $\Gamma'\cap\Gamma'^\theta$, and hence $s_1'$ is also contained in it. Then also $r_1',L_1,M_1,K_1,M_2'\in\Gamma'\cap\Gamma'^\theta$. Hence $\Gamma'\subseteq\Gamma'\cap\Gamma'^\theta$, and we conclude that $\Gamma'^\theta=\Gamma'$. This shows $(ii)$. 

Now, by \cite[Theorem~6.10]{PVM:21} we know that $\theta^3$ fixes $\Gamma'$ pointwise. As in the last line of the proof of Theorem~\ref{regnocol}, we conclude that $\theta^3$ is the identity, proving~$(i)$.

Now let $\Gamma'$ be any  full subhexagon and let $x$ be any point in it, not contained in $O$. Let $a\in O$ be collinear to $x$.  Then $x^\theta$ is collinear to $a$, and it readily follows that $a\in\Gamma'$ (since either $x^\theta$ lies on $xa$, and then by fullness, $a\in xx^\theta\subseteq\Gamma'$, or not, and then $a$ is the unique point collinear to both $x$ and $x^\theta$). Hence $O\cap \Gamma'$ is an ovoid of $\Gamma'$. This proves $(iii)$.

Now $(iv)$ follows from the fact that every ovoid in a dual split Cayley hexagon fixed by a domestic collineation is a Hermitian ovoid and hence closed under taking imaginary lines. 
\end{proof}

\begin{proof}[Geometric proof of Corollary~\ref{cor:noovoidsH2}]
Let $\Gamma$ be of class (H2). Suppose that $\theta$ is a collineation of $\Gamma$ fixing precisely an ovoid~$O$. By Proposition~\ref{ovoidsgen}$(ii)$ and $(iii)$, there exists a full proper dual split Cayley subhexagon $\Gamma'$ stabilised by $\theta$ and such that $O'=\Gamma'\cap O$ is an ovoid of $\Gamma'$. Let $p\in O\setminus O'$. Then, by Proposition~\ref{largesubh} there is a unique point $t\in\Gamma$ collinear to $p$. Since both $p$ and $\Gamma'$ are stabilised by $\theta$, so too is $t$, a contradiction.  
\end{proof}

\goodbreak

The following Corollary shows that all point-domestic collineations can be uniformly described using the setup of Theorem~\ref{thm:allexamples}.

\begin{cor}\label{cor:allexamples} Let $\Gamma$ be a Moufang hexagon with hexagonal system $(\JJ,\FF,\#)$. Let $\theta=hx_1(1)s_1$ with $h$ an automorphism of $(\JJ,\FF,\#)$.
\begin{compactenum}[$(1)$]
\item If $\Gamma$ is of class~$(\mathrm{H1})$ and $h=1$ then $\theta$ is point-domestic. Moreover
\begin{compactenum}[$(a)$]
\item if $X^2+X+1$ is irreducible over $\FF$ then $\theta$ fixes an ovoid;
\item if $X^2+X+1$ is not irreducible over $\FF$ then $\theta$ fixes a large full subhexagon.
\end{compactenum}
\item If $\Gamma$ is of class~$(\mathrm{H4})$ and $h=1$ then $\theta$ is point-domestic and fixes a ball of radius $3$ in the incidence graph centred at a line.
\item If $\Gamma$ is a triality hexagon of type ${^3}\sD_4$ and $h=\sigma$ is a nontrivial element of the Galois group then $\theta$ is point-domestic and fixes a large full subhexagon. 
\end{compactenum}
This gives the complete list of point-domestic collineations of Moufang hexagons.
\end{cor}

\begin{proof}
In each case the given element $h$ satisfies the conditions of Theorem~\ref{thm:allexamples}, and hence $\theta=hx_1(1)s_1$ is point-domestic. Claim (1) follows from \cite[Theorem~6.10]{PVM:21} (it is also easy to check directly because $\theta$ fixes a chamber if and only if $X^2+X+1$ is not irreducible over $\FF$). Claim (2) follows because we have already shown that for hexagons in class (H4) it is impossible to fix an ovoid or a large full subhexagon, leaving only a ball of radius $3$ centred at a line remaining. For claim (3), the case of a ball of radius $3$ centred at a line has already been eliminated, and so it suffices to show that there is a chamber fixed. Let $gB=x_1(z)s_1B$ with $z\neq 0$. We have (as in Corollary~\ref{cor:noovoidsH2})
$$
\theta gB=\sigma x_1(1)s_1x_1(z)s_1B=\sigma x_1(1)x_7(z)B=\sigma x_1(1)x_1(-z^{-1})s_1B=x_1(1-z^{-\sigma})s_1B.
$$
Thus the chamber $gB$ is fixed if and only if $z=1-z^{-\sigma}$, or equivalently $zz^{\sigma}-z^{\sigma}+1=0$. Choose any $a\in\JJ$ (a cubic Galois extension of $\FF$) with $a^{\sigma}\neq a$ and let $b=a-a^{\sigma}$. Then $b\neq 0$ with $\ssT(b)=0$. Then the element $z=-b^{\sigma^2}b^{-1}$ satisfies $zz^{\sigma}-z^{\sigma}+1=b^{\sigma^2}b^{-\sigma}+bb^{-\sigma}+1=\ssT(b)b^{-\sigma}=0$.

The fact that the list of examples is complete is a consequence of Theorem~\ref{thm:main}.
\end{proof}

\begin{remark}
In case (3) of Corollary~\ref{cor:allexamples} it is easy to see that the element $\theta=\sigma x_1(1)s_1$ is conjugate to $\sigma$ (as expected by Theorem~\ref{thm:subhex}). Since the fixed element set of $\sigma$ is clearly the dual split Cayley subhexagon, and since Theorem~\ref{thm:allexamples} implies $\sigma$ is domestic, it follows from Theorem~\ref{thm:PTM3} that the dual split Cayley hexagon is large in the ${^3}\sD_4$ triality hexagon (giving an independent proof of the ``if'' direction of Proposition~\ref{largesubh}).
\end{remark}


\bibliographystyle{plain}


\end{document}